\documentclass[twoside]{IEEEtran}
\def\numwav{l}
\def\be{\begin{equation}}
\def\ee{\end{equation}}
\def\beaN{\setlength{\arraycolsep}{0.0em}\begin{eqnarray*}}
\def\eeaN{\end{eqnarray*}\setlength{\arraycolsep}{5pt}}
\def\bea{\setlength{\arraycolsep}{0.0em}\begin{eqnarray}}
\def\eea{\end{eqnarray}\setlength{\arraycolsep}{5pt}}

\def\dm{n}

\def\Rd{\RR^\dm}
\def\Zd{\ZZ^\dm}

\def\Td{\TT^\dm}
\def\CC{\mathbb{C}}
\def\NN{\mathbb{N}}
\def\RR{\mathbb{R}}
\def\ZZ{\mathbb{Z}}
\def\TT{\mathbb{T}}

\def\ver{\vrule height 7pt depth 2pt width .85pt}
\def\onenorm#1{\,{\ver}\,#1\,{\ver}\,}

\def\set{\Gamma}
\def\mod{{\rm mod}\,}

\newtheorem{definition}{Definition}

\newtheorem{theorem}{Theorem}
\def\beaN{\setlength{\arraycolsep}{0.0em}\begin{eqnarray*}}
\def\eeaN{\end{eqnarray*}\setlength{\arraycolsep}{5pt}}
\def\bea{\setlength{\arraycolsep}{0.0em}\begin{eqnarray}}
\def\eea{\end{eqnarray}\setlength{\arraycolsep}{5pt}}

\def\be{\begin{equation}}
\def\ee{\end{equation}}

\def\dm{n}

\def\Rd{\RR^\dm}
\def\Zd{\ZZ^\dm}

\def\Td{\TT^\dm}
\def\CC{\mathbb{C}}
\def\NN{\mathbb{N}}
\def\RR{\mathbb{R}}
\def\ZZ{\mathbb{Z}}
\def\TT{\mathbb{T}}
\def\ver{\vrule height 7pt depth 2pt width .85pt}
\def\sver{\vrule height 5pt depth .5pt width .66pt}
\def\onenorm#1{\,{\ver}\,#1\,{\ver}\,}
\def\onenorms#1{\,{\sver}\,#1\,{\sver}\,}

\def\du{{\rm d}}
\def\disp{\displaystyle}

\def\bks{\backslash}

\def\alp{\alpha}                

\def\gam{\gamma}

\def\ome{\omega}                

\usepackage[noadjust]{cite}
\usepackage{graphicx}
\usepackage{amssymb}
\usepackage{enumerate}
\usepackage{bm,bbm}
\usepackage{subfigure}
\usepackage{booktabs}
\usepackage{tabularx}
\usepackage{array}
\usepackage{tikz}

\title{Coset Sum: an alternative to the tensor product in wavelet construction}
\author{Youngmi~Hur and Fang~Zheng
\thanks{This work was supported in part by the National Science Foundation under Grant DMS-111590.}
\thanks{This paper was presented in part at the Asilomar Conference on Signals, Systems, and Computers, Pacific Grove, CA, November 2011, at the Joint Mathematics Meetings, Boston, MA, January 2012, at the SIAM Conference on Applied Linear Algebra, Valencia, Spain, June 2012, and at the International Conference on Wavelets and Applications, St. Petersburg, Russia, July 2012.}
\thanks{The authors are with the Department of Applied Mathematics and Statistics, The Johns Hopkins University, Baltimore, MD 21218-2682 USA (e-mail: hur@jhu.edu; fzheng2@jhu.edu).}}

\pubid{Copyright (c) 2012 IEEE. Personal use is permitted. For any other purposes, permission must be obtained from the IEEE by emailing pubs-permissions@ieee.org.}
\begin{document}

\maketitle

\begin{abstract}
A multivariate biorthogonal wavelet system can be obtained from a pair of
multivariate biorthogonal refinement masks in Multiresolution Analysis
setup. Some multivariate refinement masks may be decomposed into lower
dimensional refinement masks. Tensor product is a popular way
to construct a decomposable multivariate refinement mask from lower
dimensional refinement masks.

We present an alternative method, which we call coset sum, for
constructing multivariate refinement masks from univariate refinement
masks. The coset sum shares many essential features of
the tensor product that make it attractive in practice:
(1) it preserves the biorthogonality of univariate refinement masks,
(2) it preserves the accuracy number of the univariate refinement mask, and
(3) the wavelet system associated with it has fast algorithms for
computing and inverting the wavelet coefficients.
The coset sum can even provide a wavelet system with faster algorithms in certain cases than the tensor product.
These features of the coset sum suggest that it is worthwhile to
develop and practice alternative methods to the tensor product
for constructing multivariate wavelet systems. 
Some experimental results using 2-D images are presented to illustrate our findings.
\end{abstract}


\section{Introduction}
\label{S:intro}
One of the most common tools for constructing wavelets is Multiresolution Analysis (MRA) \cite{Ma1}. In MRA, a multivariate biorthogonal wavelet system can be obtained from a pair of multivariate biorthogonal refinement masks. The tensor product has been the prevailing method for
deriving a pair of multivariate biorthogonal refinement masks from a pair of biorthogonal univariate refinement masks. 

In this paper we are interested in studying the operators that map lower dimensional refinement masks to higher dimensional refinement masks. Throughout this paper, the multidimensional (multi-D) refinement masks that can be decomposed into lower dimensional refinement masks by such operators are referred to as {\it decomposable}. One such operator is the tensor product. The multi-D refinement masks obtained via tensor product are called {\it tensor product (or separable) refinement masks}. Since the word ``separable'' is reserved for the tensor product by the definition in the literature, we use the word ``decomposable'' to indicate more general case than the tensor product.
It should be noted that a ``nonseparable'' refinement mask only means it is
not a tensor product refinement mask, and it can still be a ``decomposable'' refinement mask. Tensor product can also be used to construct multi-D wavelet masks, which are called {\it tensor product (or separable) wavelet masks} (cf.~\S\ref{subS:tensorproduct}).

In MRA setup, construction of multi-D biorthogonal wavelet systems can be done by two steps: (i) construction of multi-D biorthogonal refinement masks (or refinable functions); (ii) construction of multi-D wavelet masks.
To construct a nonseparable multi-D wavelet system, one can try making the refinement masks nonseparable in step (i) or making wavelet masks nonseparable in step (ii). Since, once a pair of multivariate biorthogonal refinement masks are given, the matrix extension problem of finding wavelet masks can always be solved by using Quillen-Suslin theorem (see, for example, \cite{RiS2}), the main effort so far for constructing nonseparable wavelets has been made in step (i). However, we note that Quillen-Suslin theorem serves only as a guide since in the process of determining the wavelet masks, some parameters still need to be specified. 

\pubidadjcol % this must be issued somewhere in the second column, when \pubid is used

Although there have been many methods for constructing nonseparable multi-D wavelets
\cite{KV1,CS,KV2,MX,RiS1,HeL,JRS,SS,CMPX,CMX,HYY,QLCCR,YDCC,Z}, 
constructing nonseparable multi-D wavelet systems is highly nontrivial.
Many of these methods work only for low spatial dimensions (2-D or 3-D) and they cannot be easily extended to other dimensions. Others assume that the wavelets or refinable functions have a special form (e.g. the refinable function has a box spline factor) and cannot be easily generalized to other cases.

\begin{table*}[t]
\caption{Comparison between tensor product and coset sum (spatial dimension $\dm\ge 2$)}
\centering
\begin{tabularx}{\textwidth}{XX}
\toprule
\textbf{Tensor product $\mathcal{T}_\dm$} ($R$, $\tilde{R}$: univariate refinement masks)& \textbf{Coset sum $\mathcal{C}_\dm$} ($R$, $\tilde{R}$: univariate refinement masks; $\tilde{R}$: {\it interpolatory}) \\[1ex]
\midrule
$\mathcal{T}_\dm[R]$ can be decomposed into the {\it product} of $R$&
$\mathcal{C}_\dm[R]$ can be decomposed into the {\it sum} of $R$\\[1ex]
\midrule
$\mathcal{T}_\dm[R]$ is interpolatory iff $R$ is interpolatory&
$\mathcal{C}_\dm[R]$ is interpolatory iff $R$ is interpolatory\\[1ex]
\midrule
$\mathcal{T}_\dm[R]$ and $\mathcal{T}_\dm[\tilde{R}]$ are biorthogonal iff $R$ and $\tilde{R}$ are biorthogonal &
$\mathcal{C}_\dm[R]$ and $\mathcal{C}_\dm[\tilde{R}]$ are biorthogonal iff $R$ and $\tilde{R}$ are biorthogonal\\[1ex]
\midrule
$\mathcal{T}_\dm[R]$ and $R$ have the same accuracy number&
$\mathcal{C}_\dm[\tilde{R}]$ and $\tilde{R}$ have the same accuracy number\\[1ex]
\midrule
$\mathcal{T}_\dm[R]$ can be decomposed into {\it non-univariate} refinement masks &
$\mathcal{C}_\dm[R]$ can be decomposed {\it only} into {\it univariate} refinement masks \\[0.5ex]
\midrule
Complexity constant in associated wavelet algorithm {\it increases with} $\dm$  &
Complexity constant in associated wavelet algorithm is {\it independent of} $\dm$ \\[0.5ex]
\bottomrule
\end{tabularx}
\label{table:tensor&coset}
\end{table*}

One of the disadvantages of the above approaches for constructing nonseparable wavelets is that they construct a pair of multi-D biorthogonal refinement masks essentially from scratch, which can be quite complicated, especially for high spatial dimensions. A simpler way to obtain multi-D biorthogonal refinement masks is to use an operator that maps $1$-D biorthogonal refinement masks to multi-D biorthogonal decomposable refinement masks. Most of the existing nonseparable wavelet construction methods (e.g. \cite{TK,PKVA,BS,SK,AL,KS,Sha}) that use decomposable refinement masks employ operators such as the McClellan transform for quincunx or other $2$-channel sampling lattices. 

Most multi-D wavelet systems that are used in practice nowadays are separable wavelet systems constructed by the tensor product of 1-D wavelet systems.
In \S\ref{subS:tensorproduct} we briefly discuss the use of tensor product in constructing biorthogonal wavelet systems. As we can see from there, 
the tensor product construction of wavelet systems is extremely simple.
This is one of the major reasons the tensor product has been so popular in constructing multi-D wavelets in practice.
However the separable wavelet systems have limitations:  (i) they have a strong directional bias along lines parallel to the coordinate directions, (ii) they are not very local\footnote{One way to measure the localness of a wavelet system is to compute the sum of the volumes of the supports of its mother wavelets (cf.~\cite{HR2,HR3}).}.

Our goal in this paper is to present an alternative method to the tensor product for constructing decomposable multi-D refinement masks. 
We call the new method as {\it coset sum}.
We show that, under an appropriate circumstance, the coset sum shares many attractive features of the tensor product. First, it preserves the biorthogonality of univariate refinement masks. Second, it preserves the accuracy number of the univariate refinement mask.
Third, it has a corresponding wavelet system which has fast algorithms for computing and inverting the wavelet coefficients. In fact, it turns out that these algorithms are faster, in certain cases, than the known algorithms based on tensor product wavelet systems. 

Let us elaborate on the last point in more detail. Suppose that we consider two wavelet systems that are constructed from the same pair of $1$-D biorthogonal refinement masks, by using tensor product and coset sum. For the tensor product wavelet system, the associated algorithm has complexity $(\alpha+\beta)\dm N$ (cf.~\S\ref{subS:tensorproduct}), where $\alpha$ and $\beta$ are the number of nonzero coefficients of the $1$-D lowpass filters for decomposition and reconstruction, respectively, $\dm$ is the spatial dimension, and $N$ is the size of an initial data to be analyzed. Thus, the constant in the complexity bound  (cf.\ Complexity discussion in \S\ref{subS:Algorithms} for the definition) in this case is $(\alpha+\beta)\dm$ and it grows linearly with the spatial dimension. On the other hand, as we can see from \S\ref{subS:Algorithms}, the complexity constant of the algorithm associated with the coset sum wavelet system we construct in this paper has complexity constant ${3\over 2}\alpha + 2\beta$, which is smaller than $(\alpha+\beta)\dm$ as long as $\dm\ge 2$. We note that the complexity constant for the coset sum case does not increase even if the spatial dimension increases. For more details, we refer to \S\ref{subS:Algorithms}.

The main difference between the coset sum method and the tensor product method is that a ``sum'' is used in obtaining the coset sum multi-D refinement masks instead of a ``product'' used in the tensor product refinement masks. Another difference is that, on the contrary to the tensor product case, the coset sum refinement mask cannot be decomposed into non-univariate refinement masks. Table~\ref{table:tensor&coset} summarizes the comparison between the tensor product and the coset sum. 

Some experimental results using 2-D images are included to show the potential usefulness of the coset sum wavelet systems we construct in this paper (cf.~\S\ref{subS:Experiments}). They show that our wavelet systems can be potentially useful for effectively approximating a certain class of images with strong directional content. They also reveal some of the limitations of our wavelet systems, which include the lack of rotational symmetry \cite{RRS}. For details, we refer to \S\ref{subS:Experiments}.

The rest of the paper is organized as follows.
In \S\ref{S:notation} we briefly overview some relevant concepts on wavelet construction.
In \S\ref{S:cosetsum} we introduce the coset sum method and discuss its
properties. In \S\ref{S:waveletsystems} we also introduce a particular class of coset sum wavelet systems, together with
the associated fast algorithms and some experimental results using our wavelet systems. We summarize our results and present some observations in \S\ref{S:summary}. Appendix contains technical details including all the proofs of the theorems in this paper.

\section{Preliminaries}
\label{S:notation}
In this section we review some relevant concepts.
\subsection{Refinement masks and wavelet masks}
\label{subS:masks}
In this paper we refer to a Laurent trigonometric polynomial as 
a {\it mask}, and a mask $\tau$ with $\tau(0)=1$ as a {\it refinement mask}. Refinement masks can be used to obtain refinable functions (see, for example, \cite{CDM}), which can in turn be used to construct wavelet systems \cite{Ma1}. 

Refinement masks $\tau$ and $\tau^\du$ are {\it biorthogonal} if they satisfy the following biorthogonal relation:
\be
\label{eq:refinebiormask}
\sum_{\gam\in\pi\set}(\overline{\tau}\tau^\du)(\ome+\gam)=1,\quad \forall\ome\in\Td:=[-\pi,\pi]^\dm,
\ee
where ${\set}:=\{0,1\}^\dm$ and the overline is used to denote the
complex conjugate. In this case, we refer to $\tau$ and $\tau^\du$ as {\it primal and dual refinement masks}, respectively.

A refinement mask
$\tau$ is {\it interpolatory} if the condition
$$
%\label{eq:interpmask}
\sum_{\gam\in\pi\set} \tau(\ome+\gam)=1
$$
holds. Thus refinement masks $\tau$ and $\tau^\du$ are biorthogonal if and only if $\overline{\tau}\tau^\du$ is interpolatory. Interpolatory masks are
widely used in subdivision schemes and wavelet constructions (for
example, see \cite{HJ} and references therein).

In this paper we say that  a filter $h:\Zd\to\RR$ is {\it associated with} a mask $\tau$ if $h$ and $\tau$ are connected via the relation $\tau(\ome)={1\over 2^\dm} \sum_{k\in\Zd} h(k)e^{-ik\cdot\ome}$ for $\ome\in\Td$.

It is straightforward to see that $\tau$ is interpolatory if and only if the associated filter $h$ satisfies
\be
\label{eq:interpfilter}
h(k)
  =\cases{1,& if $k=0$,\cr
          0,& if $k\in 2\Zd\bks 0$,\cr}
\ee
to which we refer as the interpolatory condition for the filter.

For a refinement mask $\tau$, the number of zeros of $\tau$ at $\gam\in\pi\set'$ with $\set':=\set\bks 0=\{0,1\}^\dm\bks0$
is referred to as the {\it accuracy number} \cite{SN}. Throughout the paper we assume that all refinement masks have at least accuracy number one,
since almost all of the refinement masks used in practice satisfy this condition.

We recall that the Laurent polynomials $\{t_j, t_j^\du:j=1,\cdots,\numwav\}$ are called the {\it wavelet masks} associated with a pair of biorthogonal refinement masks $(\tau, \tau^\du)$ if they satisfy the Mixed Unitary Extension Principle (MUEP) conditions \cite{RS2}: for every $\ome\in\Td$,
\be
\overline{\tau(\ome+\gam)}\tau^\du(\ome)
+\sum_{j=1}^{\numwav}\overline{t_j(\ome+\gam)}t^\du_j(\ome)
 =\cases{
1,&if $\gam=0$,\cr 0,&if $\gam\in \pi\set'$.\cr}\label{eq:mUEP}
\ee
%\bea
%\overline{\tau(\ome+\gam)}\tau^\du(\ome)
%+\sum_{j=1}^{\numwav}\overline{t_j(\ome+\gam)}t^\du_j(\ome)\nonumber\\
% =\cases{
%1,&if $\gam=0$,\cr 0,&if $\gam\in \pi\set'$.\cr}\label{eq:mUEP}
%\eea
We refer to $t_j$, $j=1,\cdots,\numwav$, and $t_j^\du$, $j=1,\cdots,\numwav$, as {\it primal and dual wavelet masks}, respectively. When $\numwav=2^\dm-1$, the masks that satisfy the MUEP conditions can be used to construct biorthogonal wavelet systems. We refer to such $(\tau,(t_j)_{j=1,\cdots,2^\dm-1})$ and
$(\tau^\du,(t_j^\du)_{j=1,\cdots,2^\dm-1})$ as
the {\it combined biorthogonal masks}. A (MRA-based) {\it biorthogonal wavelet system} is then obtained from these combined biorthogonal masks, under some simple additional conditions \cite{CDF,HR4}.

For a wavelet mask $t$, the number of zeros of
$t$ at $\omega=0$ is
referred to as the {\it number of (discrete) vanishing moments} \cite{CHR}.
It is well known (see, for example, \cite{CHR}) that for the combined biorthogonal masks
$(\tau,(t_j)_{j=1,\cdots,2^\dm-1})$ and $(\tau^\du,(t_j^\du)_{j=1,\cdots,2^\dm-1})$ whose refinement masks have at least $m$ accuracy, every primal wavelet mask $t_j$ and dual wavelet mask $t_j^\du$, $j=1,\cdots, 2^\dm-1$, has at least $m$ vanishing moments. The number of vanishing moments is closely related to the approximation performance of the wavelet system \cite{Me}.

\subsection{Tensor product wavelet construction}
\label{subS:tensorproduct}
We recall that the $\dm$-D tensor product (or separable) refinement mask  from $\dm$ (possibly distinct) univariate refinement masks $R_1, R_2, \cdots, R_\dm$ can be written as,
for  $\ome=(\ome_1,\ome_2,\cdots,\ome_\dm)\in\Td$,
\be
\label{eq:tensorproductrefinementmask}
\mathcal{T}_\dm[R_1,R_2,\cdots,R_\dm](\ome):=R_1(\ome_1)R_2(\ome_2)\cdots R_\dm(\ome_\dm).
\ee
When $R=R_1=R_2=\cdots=R_\dm$, we also use the notation $\mathcal{T}_\dm[R]$.
If we let $H$ and $h$ be the filters associated with the masks $R$ and
$\mathcal{T}_\dm[R]$ respectively, they satisfy, for $k=(k_1,k_2,\cdots,k_\dm)\in\Zd$,
$$
h(k)=H(k_1)H(k_2)\cdots H(k_\dm).
$$
It is well known that the $n$-D refinement masks constructed using tensor product preserve many useful properties of univariate refinement masks. For example, if we let $R$ and $\tilde{R}$ be univariate refinement masks, then
\begin{enumerate}[(i)]
\item $\mathcal{T}_\dm[R]$ is interpolatory if and only if $R$ is interpolatory,
\item $\mathcal{T}_\dm[R]$ and $\mathcal{T}_\dm[\tilde{R}]$ are biorthogonal if and only if $R$ and $\tilde{R}$ are biorthogonal,
\item $\mathcal{T}_\dm[R]$ and $R$ have the same accuracy number.
\end{enumerate}
Now we pose the following question. Can we find another method that
satisfies all of the above properties? An affirmative answer is
provided by the coset sum, which we introduce and study in the next section.
Before introducing the coset sum, let us review the usual approach for constructing biorthogonal wavelet systems.

Construction of 1-D biorthogonal wavelet systems is well understood. Given a pair of 1-D biorthogonal refinement masks $S_0$ and $U_0$, one sets the wavelet masks as
\be
\label{eq:1Dwaveletmasks}
S_1(\ome):=e^{-i\ome}\overline{U_0(\ome+\pi)},\quad U_1(\ome):=e^{-i\ome}\overline{S_0(\ome+\pi)}
\ee
for $\ome\in\TT$. Then the univariate pairs $(S_0,S_1)$ and $(U_0,U_1)$ satisfy the MUEP conditions (cf.~(\ref{eq:mUEP})) \cite{CDF}.

On the other hand, given a pair of multivariate biorthogonal refinement masks, constructing a multivariate biorthogonal wavelet system is not so trivial since one needs to find $2^\dm-1$ primal wavelet masks $t_j$'s and $2^\dm-1$ dual wavelet masks $t_j^\du$'s.

The usual construction of multi-D biorthogonal wavelet systems is done by the tensor product. Given a pair of 1-D biorthogonal refinement masks $S_0$ and $U_0$, one sets the $\dm$-D refinement masks as
$$\tau:=\mathcal{T}_\dm[S_0],\quad \tau^\du:=\mathcal{T}_\dm[U_0]$$
and the $\dm$-D wavelet masks as
$$t_\nu=\mathcal{T}_\dm[S_{\nu_1},S_{\nu_2},\cdots,S_{\nu_n}],\quad
t^\du_\nu=\mathcal{T}_\dm[U_{\nu_1},U_{\nu_2},\cdots,U_{\nu_n}]$$
for all $\nu=(\nu_1,\nu_2,\cdots,\nu_n)\in\set'$. Then the two refinement masks $\tau$ and $\tau^\du$ are also biorthogonal, and $(\tau,(t_\nu)_{\nu\in\set'})$ and $(\tau^\du,(t_\nu^\du)_{\nu\in\set'})$ satisfy the MUEP conditions (cf.~(\ref{eq:mUEP})). Here $\set'=\{0,1\}^\dm\bks0$ is used as before, and the univariate masks $S_1$ and $U_1$ are the ones defined in (\ref{eq:1Dwaveletmasks}). The biorthogonal wavelet systems obtained from these masks are called {\it tensor product (or separable) wavelet systems.}

It is well known that tensor product wavelet systems have fast algorithms for computing and inverting wavelet coefficients (see, for example, \cite{Ma2}), to which we refer as the {\it fast tensor product wavelet algorithms}. These algorithms have linear complexity $O(N)$, where $N$ is the size of the input data.
More precisely, if $\alp$ is the number of nonzero entries of the filter associated with $S_0$ and $\beta$ is the number of nonzero entries of the filter associated with $U_0$, then the algorithms for computing and inverting the corresponding tensor product wavelet coefficients have complexity $(\alp+\beta)nN$, where $n$ is the spatial dimension. In particular, the constant in the complexity bound is $(\alp+\beta)n$ and it increases linearly as the spatial dimension increases. 

%More precisely, the algorithm for the tensor product wavelet system with $k$ number of vanishing moments (e.g. Daubechies wavelet system of order $k$ \cite{Da}) has complexity $CknN$, where $n$ is the spatial dimension, and $C$ is a constant independent of $k, n$, and $N$. In particular, the constant in the complexity bound the complexity constant (cf.\ {\bf Complexity} in \S\ref{subS:Algorithms} for the definition) in this case is $Ckn$ and it grows linearly with the spatial dimension. 

\section{Coset sum}
\label{S:cosetsum}
\subsection{Introduction to coset sum}
\label{subS:cosetsumrefinementmasks}
We present an alternative method, called {\it coset sum}, to the tensor product in wavelet construction. Instead of the ``product'' in the tensor product, we propose to use a ``sum'' to construct multivariate refinement masks from univariate refinement masks.

Let $R$ be a univariate refinement mask and let $H$ be the univariate filter associated with $R$. For $\nu\in\set'$, the map
$$\Td\to\CC:\ome\mapsto {1\over 2^{\dm-1}}R(\ome\cdot\nu),$$
where $\ome\cdot\nu$ is the inner product in $\Rd$, is an $\dm$-D Laurent trigonometric polynomial. The normalization factor ${1\over 2^{\dm-1}}$ is used to place $R(\ome\cdot\nu)$ in the $\dm$-D space. In terms of filters, the above can be understood as aligning the 1-D filter $H$ along the $\nu$ direction:
$$\Zd\to\RR:k\mapsto \cases{H(K),&\mbox{if $k=K\nu$ for some $K\in\ZZ$
},\cr
         0,             & \mbox{otherwise}
\cr}
$$
Since we want to consider all the directions in $\set'$, a possible
candidate for the coset sum definition can be given as
$$
\Td\to\CC:\ome\mapsto A+{1\over
2^{\dm-1}}\sum_{\nu\in\set'}R(\ome\cdot\nu).
$$
Since we want the coset sum to map a 1-D {\it refinement} mask to an $\dm$-D {\it refinement} mask, by plugging in $\ome=0$, we obtain $A=-1+{1\over 2^{\dm-1}}$ and get to the following definition. 
\begin{definition}
\label{def:cosetsum}
We define the coset sum
$\mathcal{C}_\dm$ that maps a 1-D refinement mask
$R$ to an $\dm$-D refinement mask $\mathcal{C}_n[R]$ as follows:
for $\ome\in\Td$
$$
\mathcal{C}_\dm[R](\ome)
:=
{1\over
2^{\dm-1}}\left(1-2^{\dm-1}
+\sum_{\nu\in\set'}R(\ome\cdot\nu)\right),
$$
where $\set'=\set\bks0=\{0,1\}^\dm\bks0$.
\endproof
\end{definition}

\medskip\noindent
{\bf Remark 1.}
We call the refinement mask obtained by the coset sum method as the {\it coset sum refinement mask}.
The set $\set=\{0,1\}^\dm$ used in the definition is a complete set of representatives of the distinct cosets (hence the name ``coset sum'') of the quotient group $\Zd/2\Zd$. It is easy to observe that, because $\dm$-D masks are $2\pi$-periodic, the set $\{0,1\}^\dm$ used in this paper prior to the above definition (for example, for biorthogonality condition, interpolatory condition, definition of accuracy number, and MUEP conditions) can be replaced, without changing the meaning of the statements, by any other complete set of representatives of the distinct cosets of the quotient group $\Zd/2\Zd$ as long as the set contains $0$. As a result, the set $\{0,1\}^\dm$ used in the above coset sum definition can be replaced by any such an alternative set. The set $\{0,1\}^\dm$ is chosen for the discussion in this paper (with the exception of Example 3 below and discussions in \S\ref{subS:Experiments}) because it makes the support of the associated filter the smallest. Depending on applications, choosing a different set of representatives can make more sense. We emphasize that even if all the results in our paper (including Theorem \ref{thm:main}, \ref{thm:cosetsumwavelet}, and the fast coset sum wavelet algorithms in later part of the paper) are presented using the set $\set=\{0,1\}^\dm$, they will stay intact for other choices for $\set$.
\endproof

\medskip\noindent
{\bf Remark 2.}
We recall that the sum in the left-hand side of the biorthogonality condition in (\ref{eq:refinebiormask}) is taken over the set $\pi\set$, which can be considered as a set of coset representatives of $2\pi({1\over 2}\Zd/\Zd)$. The set of coset representatives has been previously used in the wavelet literature, mostly in relation with this biorthogonality condition. For example, a new algorithm called a coset by coset (CBC) is proposed in \cite{Han} for obtaining dual masks with arbitrary number of accuracy given an interpolatory primal mask, and the coset representatives are used in \cite{N} for an explicit, flexible, and easy implementation of interpolatory subdivision schemes. 
\endproof

\medskip
The coset sum for the first few low dimensions are given as follows:
$$\mathcal{C}_1[R](\ome_1)=R(\ome_1),$$
$$\mathcal{C}_2[R](\ome_1,\ome_2)
 ={1\over 2}\left\{-1+R(\ome_1)+R(\ome_2)+R(\ome_1+\ome_2)\right\},$$
\beaN
&&\mathcal{C}_3[R](\ome_1,\ome_2,\ome_3)
={1\over4}\{-3+R(\ome_1)+R(\ome_2)+R(\ome_1+\ome_2)\\
&&+R(\ome_3)+R(\ome_1+\ome_3)+R(\ome_2+\ome_3)+R(\ome_1+\ome_2+\ome_3)\}.
%&&={1\over4}\{-3+R(\ome_1)+R(\ome_2)+R(\ome_3)\\
%&&+R(\ome_1+\ome_2)+R(\ome_2+\ome_3)+R(\ome_3+\ome_1)+R(\ome_1+\ome_2+\ome_3)\}
\eeaN
%\beaN
%&&\mathcal{C}_4[R](\ome_1,\ome_2,\ome_3,\ome_4)
%={1\over8}\{-7+R(\ome_1)+R(\ome_2)+R(\ome_3)\\
%&&+R(\ome_1+\ome_2)+R(\ome_2+\ome_3)+R(\ome_3+\ome_1)+R(\ome_1+\ome_2+\ome_3)\\
%&&+R(\ome_4)+R(\ome_1+\ome_4)+R(\ome_2+\ome_4)+R(\ome_3+\ome_4)\\
%&&+R(\ome_1+\ome_2+\ome_4)+R(\ome_2+\ome_3+\ome_4)+R(\ome_3+\ome_1+\ome_4)\\
%&&+R(\ome_1+\ome_2+\ome_3+\ome_4)\}
%\eeaN

%&=&{1\over
%2^{\dm-1}}\Bigg(1-2^{\dm-1}+(2^\dm-1){H(0)\over 2}+\sum_{\nu\in\set'}\left(R(\ome\cdot\nu)-{H(0)\over 2}\right)\Bigg),\nonumber
%and $H$ is the univariate filter associated with the mask $R$.

We note that the coset sum formula in the above definition can also be written as
\be
\label{eq:simpleform}
-1+{1\over 2^{\dm-1}}\sum_{\nu\in\set}R(\ome\cdot\nu)
\ee
or
\be
\label{eq:specialform}
{1\over 2^{\dm-1}}\left({\disp 1\over \disp 2}
+\sum_{\nu\in\set'}\left(R(\ome\cdot\nu)-{\disp 1\over \disp
  2}\right)\right).
\ee

The filter $h$ associated with the coset sum refinement mask $\mathcal{C}_\dm[R]$ is connected to the univariate filter $H$ via
\be
\label{eq:filter}
h(k)=
\cases{H(K),\quad\mbox{if $k=K\nu$ for some $K\in\ZZ\bks 0,
\nu\in\set'$},\cr
2^\dm-(2^\dm-1)(2-H(0)),\quad\mbox{if $k=0$},\cr
          0,           \quad\mbox{for all other $k\in\Zd$}.\cr}
\ee

If the univariate filter $H$ associated with $R$ is interpolatory, 
the $\dm$-D filter $h$ associated with $\mathcal{C}_\dm[R]$ is also interpolatory and it can be expressed as
$$
h(k)=\cases{H(K),&\mbox{if $k=K\nu$ for some $K\in\ZZ,
\nu\in\set'$},\cr
         0,             & for all other \mbox{$k\in\Zd$}.\cr}
$$
In particular, the restriction of the $\dm$-D filter $h$ to $\nu$ direction, for each $\nu\in\set'$, is the 1-D filter $H$. 

Now we give a few very simple examples of constructing multi-D refinement filters from univariate refinement filters.

\begin{figure}[t]
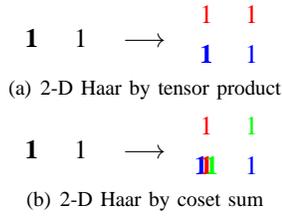

\centering
\subfigure[2-D Haar by tensor product]{
~~~~~{\bf1} \quad1 \quad
$\longrightarrow$~
\begin{tabular}{cc}
\textcolor{red}   1& \textcolor{red}1~~~~~\\[3pt]
\bf\textcolor{blue} 1& \textcolor{blue} 1~~~~~\\[3pt]
\end{tabular}
}
\qquad
\subfigure[2-D Haar by coset sum]{
~~~~~{\bf1} \quad1 \quad
$\longrightarrow$~~~
\put(4,-6) {\textcolor{blue}{\bf1}}
\put(6,-6) {\textcolor{red}{\bf1}}
\put(8,-6) {\textcolor{green}{\bf1}}
\begin{tabular}{cc}
\textcolor{red}   1& \textcolor{green}1~~~~~\\[3pt]
 & \textcolor{blue} 1~~~~~\\[3pt]
\end{tabular}
}
\caption{Constructions of 2-D Haar refinement filter (Tensor product and Coset sum) (cf.~Example 1)}
\label{figure:Haar}
\end{figure}

\medskip\noindent
{\bf Example 1: $n$-D Haar refinement filter: the only filter that can be obtained using either the tensor product or the coset sum.} Consider the 2-D Haar refinement filter
$$
h(k)
  =\cases{1,& if $\ k=(0,0), (1,0), (0,1)$ or $(1,1)$,\cr
          0,& otherwise.\cr}
$$
Let $H$ be the 1-D Haar refinement filter 
$$
H(K)
  =\cases{1,& if $\ K=0$ or $K=1$,\cr
          0,& otherwise.\cr}
$$
Then $h$ can be obtained from $H$ either by
\begin{enumerate}[(I)]
\item (Tensor Product Case) aligning the filter $H$ along $y=0$ line ($x$-axis) and $y=1$ line (see Figure\footnote{In the figures of filters drawn in this paper, the bold-faced number is used to represent the value of the filter at the origin.}~\ref{figure:Haar}(a)), or by
\item (Coset Sum Case) aligning the filter $H$ along $y=0$ line ($x$-axis), $x=0$ line ($y$-axis), and $y=x$ line (see Figure~\ref{figure:Haar}(b)).
\end{enumerate}
Since the support of the 2-D tensor product refinement filter will always be a rectangle and the support of the 2-D coset sum refinement filter will always be the union of three line segments in different directions, it is easy to see that, up to the integer translation, the 2-D Haar refinement filter is the only 2-D filter that can be obtained using either the tensor product or the coset sum.
It is straightforward to show that, for arbitrary spatial dimension $n$, the $\dm$-D Haar refinement filter is the only filter that can be obtained using either the tensor product or the coset sum.
\endproof

\medskip\noindent
{\bf Example 2: Refinement filter associated with an $\dm$-D piecewise-linear box spline.} Let us consider the 2-D refinement filter $h$ associated with a 2-D piecewise-linear box spline \cite{BHR}:
$$
h(k)
  =\cases{1,& if $\ k=(0,0)$,\cr
  {1\over 2},& if $\ k=\pm(1,0)$, $\pm(0,1)$, or $\pm(1,1)$,\cr
          0,& otherwise.\cr}
$$
Let $H$ be the refinement filter associated with a 1-D piecewise-linear spline:
$$
H(K)
  =\cases{1,& if $\ K=0$,\cr
  {1\over 2}, & if $\  K=\pm 1$,\cr
          0,& otherwise.\cr}
$$
Then $h$ can be obtained from $H$ by
aligning the filter $H$ along $y=0$ line ($x$-axis), $x=0$ line ($y$-axis), and $y=x$ line (see Figure~\ref{figure:spline}). In other words, $h=\mathcal{C}_2[H]$. In fact, it is easy to see that for the $\dm$-D refinement filter $h$ associated with an $\dm$-D piecewise-linear box spline, we have $h=\mathcal{C}_\dm[H]$.
\endproof

\begin{figure}[t]
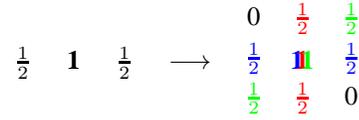

\centering
$1\over2$ \quad {\bf1} \quad $1\over2$ \quad
$\longrightarrow$~~~
\put(23,0) {\textcolor{blue}{\bf1}}
\put(25,0) {\textcolor{red}{\bf1}}
\put(27,0) {\textcolor{green}{\bf1}}
\begin{tabular}{ccc}
0                             & $\textcolor{red}{1\over2}$ & $\textcolor{green}{1\over2}$\\[3pt]
$\textcolor{blue} {1\over2}$ &                          & $\textcolor{blue} {1\over2}$\\[3pt]
$\textcolor{green}{1\over2}$ & $\textcolor{red}{1\over2}$ &   0                        \\[3pt]
\end{tabular}
\caption{Construction of 2-D piecewise-linear box spline refinement filter (Coset sum) (cf.~Example 2)}
\label{figure:spline}
\end{figure}

\medskip\noindent
{\bf Example 3: Refinement filter supported on different directed line segments.} 
We consider $\dm=2$ and choose the same univariate filter $H$ as in Example 2, but choose the 2-D filter $h$ differently:
$$
h(k)
  =\cases{1,& if $\ k=(0,0)$,\cr
  {1\over 2},& if $\ k=\pm(1,2)$, $\pm(2,1)$, or $\pm(-1,1)$,\cr
          0,& otherwise.\cr}
$$
Then $h=\mathcal{C}_2[H]$ with $\set$ chosen differently (cf.~Remark 1 after Definition~\ref{def:cosetsum}): 
$$\set=\{(0,0),(2,1),(1,2),(-1,1)\}.$$
In particular, $h$ can be obtained from $H$ by aligning the filter $H$ along $y=x/2$ line, $y=2x$ line, and $y=-x$ line (see Figure~\ref{figure:differentspline}). Note that the filter $h$ is supported on the line segments that are not parallel to the coordinate directions.  
\endproof

\subsection{Properties of coset sum refinement masks}
\label{subS:cosetsumproperties}
In this subsection, we study the properties of the multi-D refinement masks obtained by the coset sum method.

The following theorem shows that the refinement masks obtained by the coset sum share many important properties with the tensor product refinement masks.
\begin{theorem}
\label{thm:main}
Let $\mathcal{C}_\dm$ be the coset sum, and let $R$ and $\tilde{R}$ be univariate refinement masks.
\begin{enumerate}[(a)]
\item $\mathcal{C}_\dm[R]$ is interpolatory if and only if $R$ is interpolatory.
\item Suppose that one of $R$ and $\tilde{R}$ is interpolatory. Then $\mathcal{C}_\dm[R]$ and $\mathcal{C}_\dm[\tilde{R}]$ are biorthogonal if and only if $R$ and $\tilde{R}$ are biorthogonal.
\item Suppose that $R$ is interpolatory. Then $\mathcal{C}_n[R]$ and $R$ have the same accuracy number.
\end{enumerate}
\end{theorem}

\begin{proof}
See Appendix~\ref{subS:proofofThm1}.
\end{proof}

\medskip
Below we add a few remarks on Theorem~\ref{thm:main}.

\begin{figure}[t]
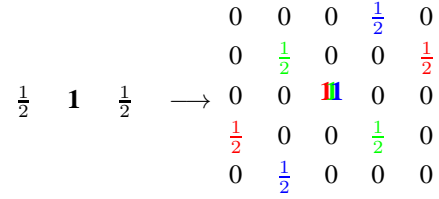

\centering
$1\over2$ \quad {\bf1} \quad $1\over2$ \quad
$\longrightarrow$
\put(41,3) {\bf\textcolor{red}1}
\put(43,3) {\bf\textcolor{green}1}
\put(45,3) {\bf\textcolor{blue}1}
\begin{tabular}{ccccc}
0&0&0&$\textcolor{blue}{1\over2}$&0\\[3pt]
0&$\textcolor{green} {1\over2}$&0&0&$\textcolor{red}{1\over2}$\\[3pt]
0&0& &0&0\\[3pt]
$\textcolor{red}{1\over2}$&0&0&$\textcolor{green} {1\over2}$&0\\[3pt]
0&$\textcolor{blue} {1\over2}$&0&0&0\\[3pt]
\end{tabular}
\caption{2-D coset sum refinement filter supported on different directed line segments (cf.~Example 3)}
\label{figure:differentspline}
\end{figure}

\medskip\noindent
{\bf Remark on Theorem~\ref{thm:main}(b).} The interpolatory condition in part (b) cannot be omitted. To see this, we consider the univariate refinement mask associated with Daubechies wavelet system of order $2$ \cite{Da}, and let
$$R(\ome)=\tilde{R}(\ome)
 =\cos^2({\ome\over 2})\left({{1+\sqrt{3}}\over 2}
 +{{1-\sqrt{3}}\over 2}e^{-i\ome}\right),\ome\in\TT.$$
Then $R$ (hence $\tilde{R}$) is not interpolatory, and 
$R$ and $\tilde{R}$ are biorthogonal. However it is easy to see that
 $\mathcal{C}_2[R]$
and
$\mathcal{C}_2[\tilde{R}]$ are not biorthogonal.
\endproof

\medskip\noindent
{\bf Remark on Theorem~\ref{thm:main}(c).} For general (not necessarily interpolatory) $R$,
the accuracy number of $\mathcal{C}_\dm[R]$ is at least $\min\{m_1,m_2\}$ where $m_1$ is the accuracy number of $R$ and $m_2$ is the order that $1-R$ has a zero at the origin. This statement can be proved using similar arguments as in the proof of Theorem~\ref{thm:main}(c), and we omit the proof.
\endproof

\medskip

The Deslauriers-Dubuc mask \cite{DD} of order $2k$
($k\in\NN$) is defined as
\bea
U_{2k}(\ome)&:=&\cos^{2k}({\ome\over 2})P_k(\sin^2({\ome\over 2})),\label{eq:defU2k}\\
P_k(x)&:=&\sum_{j=0}^{k-1}{\disp(k-1+j)!\over\disp j!(k-1)!}x^j.\nonumber
\eea
The mask $U_{2k}$ is interpolatory and has accuracy number $2k$.
We now present a family of biorthogonal coset sum refinement masks based on the Deslauriers-Dubuc interpolatory masks.

\begin{figure}[t] %t--top of page;b--bottom of page;h--here;!--force to implement
\centering
\setlength{\unitlength}{0.8mm}
\begin{picture}(75,72)
\put(3,72){
\small\begin{tabular*}{0.25\textwidth}{@{\extracolsep{\fill}}ccccccc}
{\color{blue}$-{1\over16}$}&{\color{blue}0}&{\color{blue}${9\over16}$}&{\color{blue}\bf1}&{\color{blue}${9\over16}$}&{\color{blue}0}&{\color{blue}$-{1\over16}$}\\[4pt]
\end{tabular*}
}
\put(14,66){Filter associated with $U_4$}
\put(38,59){$\mathcal{C}_2$ (Coset sum)}
\put(35,64){\vector(0,-1){10}}
\put(1,27){
\small\begin{tabular*}{0.25\textwidth}{@{\extracolsep{\fill}}ccccccc}
0&0&0&{\color{blue}$-{1\over16}$}&0&0&{\color{blue}$-{1\over16}$}\\[4pt]
0&0&0&{\color{blue}0}&0&{\color{blue}0}&0\\[4pt]
0&0&0&{\color{blue}${9\over16}$}&{\color{blue}${9\over16}$}&0&0\\[4pt]
{\color{blue}$-{1\over16}$}&{\color{blue}0}&{\color{blue}${9\over16}$}&{\color{blue}\bf1}&{\color{blue}${9\over16}$}&{\color{blue}0}&{\color{blue}$-{1\over16}$}\\[4pt]0&0&{\color{blue}${9\over16}$}&{\color{blue}${9\over16}$}&0&0&0\\[4pt]
0&{\color{blue}0}&0&{\color{blue}0}&0&0&0\\[4pt]
{\color{blue}$-{1\over16}$}&0&0&{\color{blue}$-{1\over16}$}&0&0&0\\[4pt]
\end{tabular*}
}
\put(11,0){Filter associated with $\mathcal{C}_2[U_{4}]$}
\end{picture}
\caption{Refinement filters associated with the masks $U_4$ and $\mathcal{C}_2[U_{4}]$ in Example 4}
\label{figure:C2U4}
\end{figure}

\medskip\noindent
{\bf Example 4: A family of $\dm$-D biorthogonal coset sum refinement masks.}
For each $k\in\NN$, we choose $U_{2k}$ in (\ref{eq:defU2k}) as a univariate interpolatory refinement mask.
By Theorem~\ref{thm:main}(a)(c), $\mathcal{C}_\dm[U_{2k}]$ is an $\dm$-D interpolatory refinement mask with accuracy number $2k$. It is straightforward to see that for each $k\in\NN$, 
\be
\label{eq:defS2k}
S_{2k}:=U_{2k}(3-2U_{2k})
\ee
is biorthogonal\footnote{Given a refinement filter, a dual refinement filter is not uniquely determined in general. The specific choice of the dual filter of $U_{2k}$ as in (\ref{eq:defS2k}) can be obtained, for example, from Proposition 2.1 in \cite{JS}. See also Theorem 2 in \cite{Hur} for an alternative derivation based on a critical representation of the Laplacian pyramid (\cite{BA}).} to $U_{2k}$. By Theorem~\ref{thm:main}(b),
$\mathcal{C}_\dm[U_{2k}]$ is biorthogonal to
$\mathcal{C}_\dm[S_{2k}]$. Since $S_{2k}$ has at least $2k$
accuracy and $1-S_{2k}$ has a zero of order at least $2k$ at the origin, by the Remark on Theorem~\ref{thm:main}(c),
$\mathcal{C}_\dm[S_{2k}]$ has at least
$2k$ accuracy. The filters for the case $k=\dm=2$ are depicted in
Figure~\ref{figure:C2U4} and \ref{figure:C2S4}. Using the standard tool in wavelet literature (see, for example, \cite{LLS} and references therein), one can show that both $\mathcal{C}_2[U_{4}]$ and $\mathcal{C}_2[S_{4}]$ generate the refinable functions that are in $L^2(\RR^2)$ (cf. Figure~\ref{figure:s4u4refinement}).
\endproof

\medskip
Similar to the tensor product case, the coset sum can actually take different univariate refinement masks. However, since the cardinality of the set $\set'$ is $2^\dm-1$, we have $2^\dm-1$ different directions to consider, instead of $\dm$ different coordinate directions for the tensor product case. In such a case the $\dm$-D coset sum refinement can be written as 
\be
\label{eq:cosetsumrefinementmask}
\mathcal{C}_n[(R_\nu)_{\nu\in\set'}](\ome):={1\over
2^{\dm-1}}\left(1-2^{\dm-1}
+\sum_{\nu\in\set'}R_\nu(\ome\cdot\nu)\right),
\ee
where $R_\nu$, $\nu\in\set'$, are possibly distinct univariate refinement masks for different direction $\nu$.

Let $n=n_1+n_2+\cdots+n_m,  n_j\ge1$ for $j=1,2,\cdots,m$. Then the tensor product refinement mask in (\ref{eq:tensorproductrefinementmask}) can be written as the product of possibly non-univariate lower dimensional tensor product refinement masks as follows: for $\ome=(\ome_1,\ome_2,\cdots,\ome_\dm)\in\Td$,
\beaN
&&\mathcal{T}_{\dm}[R_1,\cdots,R_n](\ome)\\
&&=\mathcal{T}_{n_1}[R_1,\cdots,R_{n_1}](\ome_1,\cdots,\ome_{n_1})\cdot\\
&&\quad\:\mathcal{T}_{n_2}[R_{n_1+1},\cdots,R_{n_1+n_2}](\ome_{n_1+1},\cdots,\ome_{n_1+n_2})\cdot\\
&&\cdots
\mathcal{T}_{n_m}[R_{n_1+\cdots+n_{m-1}+1},\cdots,R_n](\ome_{n_1+\cdots+n_{m-1}+1},\cdots,\ome_n).
\eeaN
On the contrary, the coset sum refinement mask {\it cannot} be written as the sum of non-univariate lower dimensional coset sum refinement masks.

\begin{figure}[t]
\begin{picture}(75,220)
\put(0,220){
\scalebox{0.7}
{
\centering
\small\begin{tabular}{ccccccccccccc}
{\color{blue}$-\frac{2}{512}$}&{\color{blue}0}&{\color{blue}$\frac{36}{512}$}&{\color{blue}$-\frac{32}{512}$}&{\color{blue}$-\frac{126}{512}$}&{\color{blue}$\frac{288}{512}$}& {\color{blue}$\bf\frac{696}{512}$}&{\color{blue}$\frac{288}{512}$}&{\color{blue}$-\frac{126}{512}$}&{\color{blue}$-\frac{32}{512}$}&{\color{blue}$\frac{36}{512}$}&{\color{blue}0}&{\color{blue}$-\frac{2}{512}$}\\[7pt]
\end{tabular}
}
}
\put(84,205){Filter associated with $S_4$}
\put(137,188){$\mathcal{C}_2$ (Coset sum)}
\put(132,200){\vector(0,-1){23}}
\put(1,90){
\scalebox{0.7}
{
\centering
\small\begin{tabular}{ccccccccccccc}
0&0&0&0&0&0&{\color{blue}$-\frac{2}{512}$}&0&0&0&0&0&{\color{blue}$-\frac{2}{512}$}\\[7pt]
0&0&0&0&0&0&{\color{blue}0}&0&0&0&0&{\color{blue}0}&0\\[7pt]
0&0&0&0&0&0&{\color{blue}$\frac{36}{512}$}&0&0&0&{\color{blue}$\frac{36}{512}$}&0&0\\[7pt]
0&0&0&0&0&0&{\color{blue}$-\frac{32}{512}$}&0&0&{\color{blue}$-\frac{32}{512}$}&0&0&0\\[7pt]
0&0&0&0&0&0&{\color{blue}$-\frac{126}{512}$}&0&{\color{blue}$-\frac{126}{512}$}&0&0&0&0\\[7pt]
0&0&0&0&0&0&{\color{blue}$\frac{288}{512}$}&{\color{blue}$\frac{288}{512}$}&0&0&0&0&0\\[7pt]
{\color{blue}$-\frac{2}{512}$}&{\color{blue}0}&{\color{blue}$\frac{36}{512}$}&{\color{blue}$-\frac{32}{512}$}&{\color{blue}$-\frac{126}{512}$}&{\color{blue}$\frac{288}{512}$}& {\color{blue}$\bf\frac{1064}{512}$}&{\color{blue}$\frac{288}{512}$}&{\color{blue}$-\frac{126}{512}$}&{\color{blue}$-\frac{32}{512}$}&{\color{blue}$\frac{36}{512}$}&{\color{blue}0}&{\color{blue}$-\frac{2}{512}$}\\[7pt]
0&0&0&0&0&{\color{blue}$\frac{288}{512}$}&{\color{blue}$\frac{288}{512}$}&0&0&0&0&0&0\\[7pt]
0&0&0&0&{\color{blue}$-\frac{126}{512}$}&0&{\color{blue}$-\frac{126}{512}$}&0&0&0&0&0&0\\[7pt]
0&0&0&{\color{blue}$-\frac{32}{512}$}&0&0&{\color{blue}$-\frac{32}{512}$}&0&0&0&0&0&0\\[7pt]
0&0&{\color{blue}$\frac{36}{512}$}&0&0&0&{\color{blue}$\frac{36}{512}$}&0&0&0&0&0&0\\[7pt]
0&{\color{blue}0}&0&0&0&0&{\color{blue}0}&0&0&0&0&0&0\\[7pt]
{\color{blue}$-\frac{2}{512}$}&0&0&0&0&0&{\color{blue}$-\frac{2}{512}$}&0&0&0&0&0&0\\[7pt]
\end{tabular}
}
}
\put(78,0){Filter associated with $\mathcal{C}_2[S_{4}]$}
\end{picture}
\caption{Refinement filters associated with the masks $S_4$ and $\mathcal{C}_2[S_{4}]$ in Example 4}
\label{figure:C2S4}
\end{figure}

% SPARSE VERSION
\begin{figure}[t]
\includegraphics[scale=0.23]{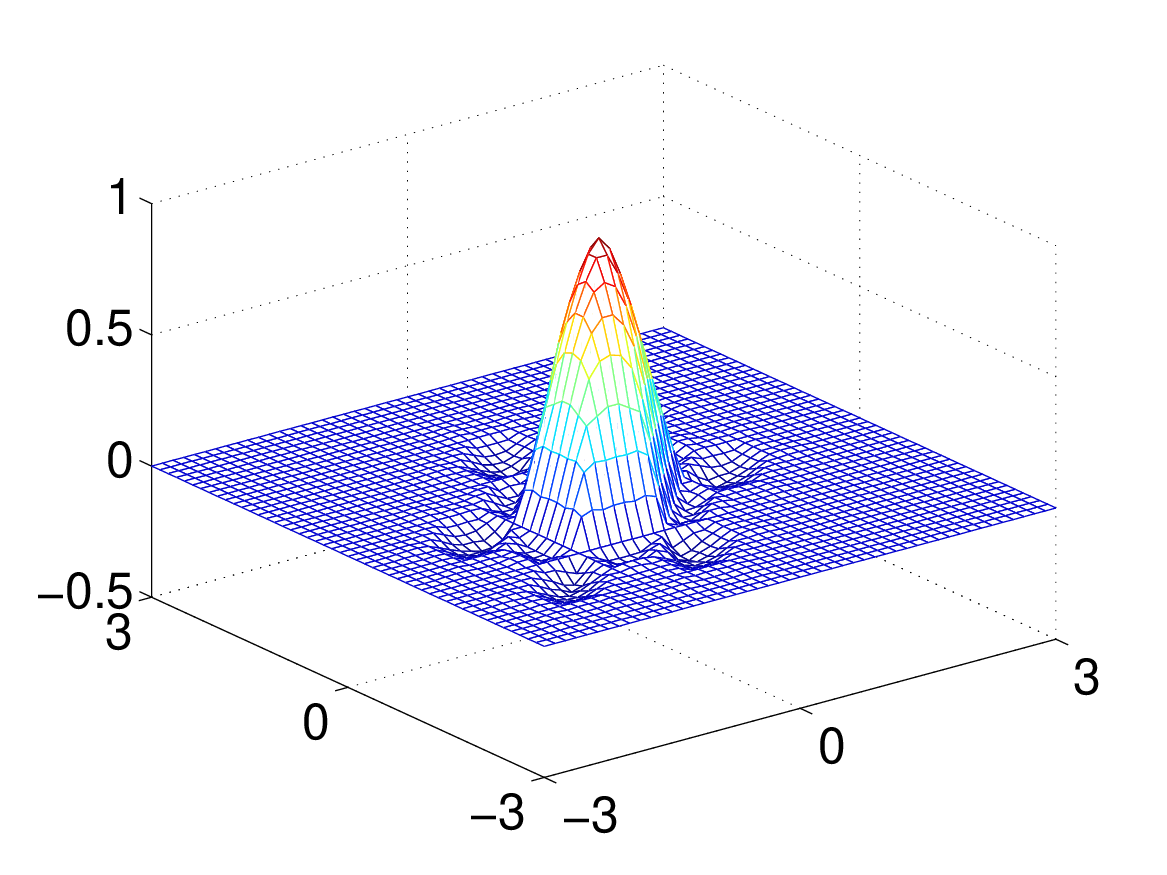}
\hspace{-0.5cm}
\includegraphics[scale=0.23]{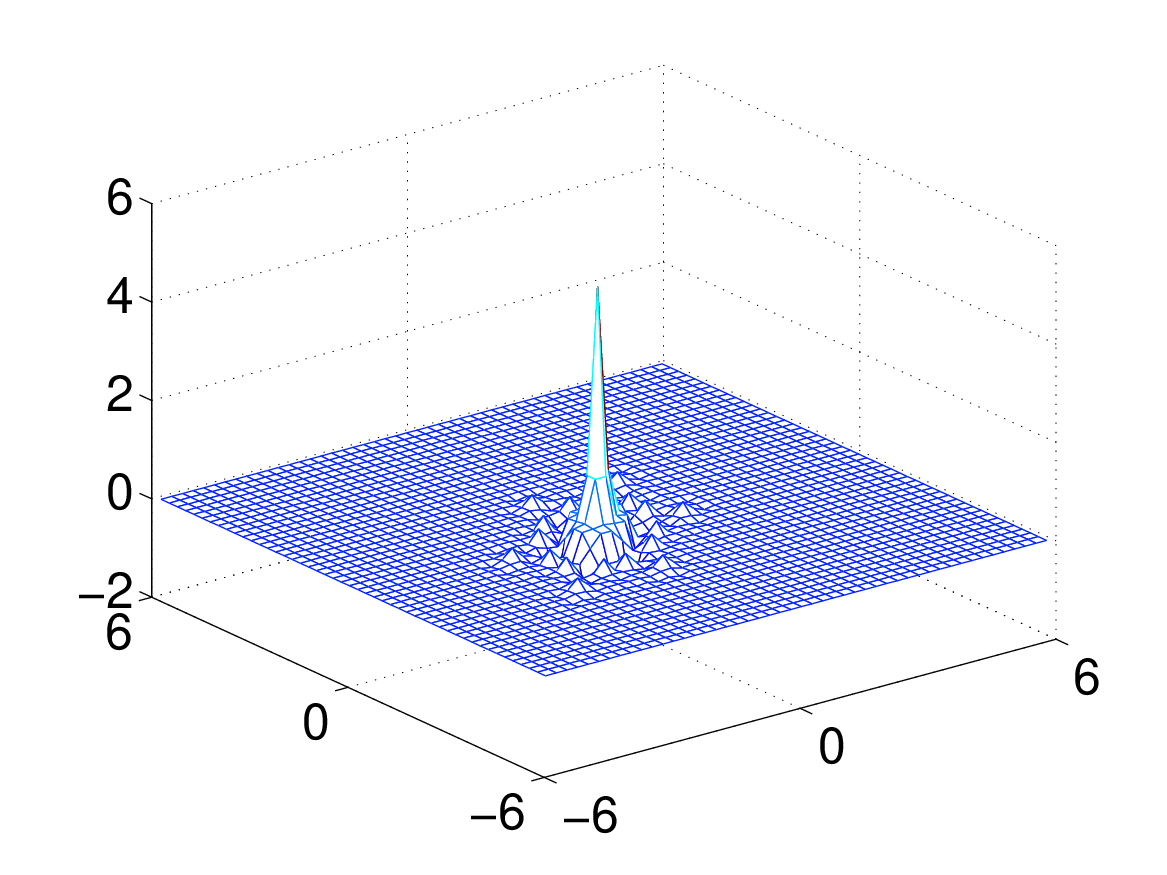}
\caption{The refinable functions associated with the coset sum refinement masks $\mathcal{C}_2[U_4]$ (left) and $\mathcal{C}_2[S_4]$ (right) in Example 4}
\label{figure:s4u4refinement}
\end{figure}

We can also consider a hybrid of the coset sum and the tensor product : for $n=n_1+n_2+\cdots+n_m,  n_j\ge1, j=1,2,\cdots,m$,
\bea
&&\mathcal{C}_{n_1}[R](\ome_1,\cdots,\ome_{n_1})\cdot\mathcal{C}_{n_2}[R](\ome_{n_1+1},\cdots,\ome_{n_1+n_2})\cdot\nonumber\\
&&\cdots\mathcal{C}_{n_m}[R](\ome_{n_1+\cdots+n_{m-1}+1},\ldots,\ome_n).\label{eq:hybridrefinementmask}
\eea
Similar statements to the ones of Theorem~\ref{thm:main} can be made for the coset sum refinement mask in a generalized sense as in (\ref{eq:cosetsumrefinementmask}) and for the hybrid refinement mask as in (\ref{eq:hybridrefinementmask}). We omit the statements and the proofs as they are similar to the ones of Theorem~\ref{thm:main}.

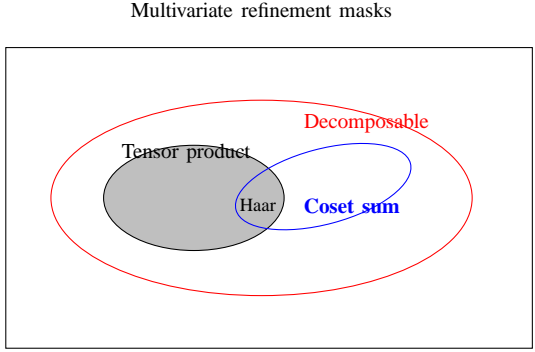
\begin{figure}[t]
\centering
\begin{tikzpicture}
\draw (-2.5,2) rectangle (4.5,-2);
\node at (0.9,2.5) {\footnotesize Multivariate refinement masks};
%\draw [fill=lightgray](0,0) circle (1.1 and 1.1);
\draw[fill=lightgray](0,0) ellipse(1.2 and 0.7);
\draw [color=blue][rotate=15](1.7,-0.3) ellipse (1.2 and 0.5);
\draw [color=red](0.9,0) ellipse (2.8 and 1.3);
\node at (-0.1,0.6) {\footnotesize Tensor product};
\node at (2.1,-0.1) [color=blue]{\bf\footnotesize Coset sum};
\node at (2.3,1) [color=red]{\footnotesize Decomposable};
%\node at (3.2,-1.7) {\textcolor {gray} {Nonseparable}};
\node at (0.85,-0.1) {\scriptsize Haar};
\end{tikzpicture}
\caption{\footnotesize  The tensor product multivariate refinement masks are not the only decomposable refinement masks. The coset sum provides a systematic way to construct other types of decomposable refinement masks. The other decomposable refinement masks include the ones constructed by the existing approaches (cf. discussion in \S\ref{S:intro}). The multivariate Haar refinement mask is essentially the only mask that can be obtained by using either the tensor product or the coset sum (cf.~Example 1).}
%\vspace{-30pt}
\label{figure:Venn diagram}
\end{figure}

The diagram in Figure~\ref{figure:Venn diagram} illustrates the relation among the tensor product, the coset sum, and the decomposable multi-D refinement masks. We note that the type of decomposable refinement masks that can be obtained by coset sum is different from the one by the aforementioned existing methods \cite{TK,PKVA,BS,SK,AL,KS,Sha} since coset sum works for $2^\dm$-channel sampling lattices (cf. \S\ref{subS:Algorithms}). 

\section{Application: coset sum wavelet systems}
\label{S:waveletsystems}
In this section we introduce a special class of wavelet systems that can be derived from coset sum refinement masks in a very simple manner, and present their properties, including fast algorithms, together with some experimental results.

\subsection{Coset sum wavelet systems}
\label{subS:cosetsumwaveletsystems}
Since the coset sum provides a way to construct a pair of multivariate biorthogonal refinement masks from univariate ones, it can be combined with any procedure for finding wavelet masks to construct a multivariate biorthogonal wavelet system. It is well known (for example, see \cite{P} and Example~5 below) that for a given pair of $\dm$-D biorthogonal refinement masks, different biorthogonal wavelet systems can be obtained by choosing wavelet masks differently. The specific choice we make in this paper is guided by the simplicity of the form of the primal wavelet masks (cf.~(\ref{eq:nDwaveletmask}) and the discussion below). Use of other criteria may result in a totally different type of ``coset sum'' wavelet systems, hence discussing about properties of coset sum wavelet systems makes sense only after the wavelet masks are specifically chosen. Below we present our approach for determining the wavelet masks.

Suppose that $S$ and $U$ are 1-D biorthogonal refinement masks, and that $U$ is interpolatory. Theorem~\ref{thm:main}(b) implies that the $\dm$-D coset sum refinement masks $\mathcal{C}_\dm[S]$ and $\mathcal{C}_\dm[U]$ are biorthogonal. Moreover, from (\ref{eq:specialform}) and the assumption that $U$ is interpolatory, we see that the restriction of the $\dm$-D mask $\mathcal{C}_\dm[U]$ to $\nu$ direction, $\nu\in\set'=\{0,1\}^\dm\bks0$, is given by $U(\ome\cdot\nu)$ for $\ome\in\Td$ (up to constants), which is essentially a 1-D mask. Hence, as in the 1-D wavelet construction (cf.~(\ref{eq:1Dwaveletmasks})), one can attempt to define the multivariate wavelet masks $t_\nu$, $\nu\in\set'$, (note that we have $2^\dm-1$ wavelet masks) of the form
\be
\label{eq:nDwaveletmask}
t_\nu(\ome)=e^{-i\ome\cdot\nu}\overline{U(\ome\cdot\nu+\pi)}, \quad \ome\in\Td.
\ee
The next theorem shows that the above approach leads to the construction of $\dm$-D biorthogonal wavelet systems. 
\begin{theorem}
\label{thm:cosetsumwavelet}
Suppose that $S$ and $U$ are 1-D biorthogonal refinement masks, and that $U$ is interpolatory. Define $\dm$-D biorthogonal refinement masks as
$$
\tau:=\mathcal{C}_\dm[S],\quad
\tau^\du:=\mathcal{C}_\dm[U],
$$
and $\dm$-D primal wavelet masks $t_\nu$, $\nu\in\set'$, as in (\ref{eq:nDwaveletmask}).
Then there exist dual wavelet masks $t_\nu^\du$, $\nu\in\set'$, such that $(\tau,(t_\nu)_{\nu\in\set'})$ and $(\tau^\du,(t_\nu^\du)_{\nu\in\set'})$ are $\dm$-D combined biorthogonal masks.
\endproof
\end{theorem}

\medskip
\begin{proof}
See Appendix~\ref{subS:proofofThm2}.
\end{proof}

\medskip\noindent
{\bf Remark 1.}
We refer to the biorthogonal wavelet system constructed from the $\dm$-D combined biorthogonal masks in Theorem~\ref{thm:cosetsumwavelet} as the {\it canonical coset sum wavelet system}. As we discussed previously, there may be many other coset sum wavelet systems associated with the same coset sum refinement masks. Throughout this paper, the word ``canonical'' is suppressed when no confusion arises.
\endproof

\medskip\noindent
{\bf Remark 2.}
The exact form of the dual wavelet masks $t_\nu^\du$, $\nu\in\set'$, of the canonical coset sum wavelet system in Theorem~\ref{thm:cosetsumwavelet} is not important for understanding our results in this paper, but knowing it may be useful in some other contexts. By carefully inspecting the proof of Theorem~\ref{thm:cosetsumwavelet}, we see that the dual wavelet masks $t_\nu^\du$, $\nu\in\set'$, have the form
\be
\label{eq:tnudu}
t_\nu^\du(\ome)=2^{-\dm+1}e^{-i\ome\cdot\nu}(1-2\tau^\du(\ome)
\overline{S^{o}(\ome\cdot\nu)}),\;\ome\in\Td,
\ee
where $S^{o}:=(S-S(\cdot+\pi))/2$ is the odd part of $S$.
\endproof

\medskip\noindent
{\bf Remark 3.}
We recall that there is a nonseparable multi-D wavelet construction
method based on the traditional lifting scheme (\cite{Sw}) proposed by
J. Kova\v{c}evi\'{c} and W. Sweldens \cite{KoSw}. A key ingredient of
their construction is a class of $\dm$-D filters called Neville
filters, which are used to build the predict and update filters. Given
suitable $\dm$-D Neville filters, their method can construct the
associated $\dm$-D
biorthogonal wavelet systems. It turns out that if the Neville filters
are extracted from the $\dm$-D
biorthogonal coset sum refinement masks, the above canonical coset sum
wavelet systems can also be obtained by their method, and our fast
algorithms associated with these wavelet systems (cf.
\S\ref{subS:Algorithms}) can be viewed as realization of a special
case of their fast transform. However it should be noted that the
Neville filters extracted from the $\dm$-D biorthogonal coset sum
refinement masks cannot be obtained from \cite{KoSw}. Also note that
other coset sum wavelet systems besides the canonical ones cannot be
constructed by their method regardless of the choice of the Neville filters.
\endproof

\medskip
Canonical coset sum wavelet systems have many potentially useful properties. The most distinctive property is that they can be associated with fast algorithms, which is explained in detail in the next subsection. Another (related) property is that they can be much more local than tensor product wavelet systems. The easiest way to see this property is probably through the following example. 

\medskip\noindent
{\bf Example 5: $\dm$-D coset sum Haar wavelet systems.} The simplest choice for the univariate refinement mask is the $1$-D Haar refinement mask $R(\ome)={1\over2}+{1\over2}e^{-i\ome}$, $\ome\in\TT$, which is biorthogonal to itself and interpolatory. Let $\tau=\tau^\du=\mathcal{C}_\dm[R]$ be the $\dm$-D Haar refinement mask, which can be obtained either by coset sum or tensor product (cf.~Example~1). Then from Theorem~\ref{thm:cosetsumwavelet} and the remarks after it, we obtain the canonical coset sum wavelet system whose $\dm$-D biorthogonal refinement masks are $\tau$ and $\tau^\du$, 
and whose primal and dual wavelet masks are, for $\nu\in\set'$ and $\ome\in\Td$,
$$
t_{\nu}(\ome)={e^{-i\ome\cdot\nu}-1\over 2},\quad
t_{\nu}^\du(\ome)={e^{-i\ome\cdot\nu}-\tau^\du(\ome)\over 2^{\dm-1}}.
$$
We refer to this wavelet system as the $\dm$-D (canonical) coset sum Haar wavelet system.
Each of the wavelet masks of this wavelet system has one vanishing moment. We note that this wavelet system is the same as the piecewise-constant biorthogonal wavelet system introduced in \cite{HR2} (up to constants), which is shown to be far more local than the tensor product Haar wavelet system, in high spatial dimensions.
\endproof

\begin{figure}[t]
\centering
\subfigure[Wavelet filter associated with the mask $t_{1,0}(\ome_1,\ome_2)$]{
\small
\centering
\begin{tabular*}{0.35\textwidth}{@{\extracolsep{\fill}}ccccccccc}
0&0&0&0&0&0&0&0&0\\[3pt]
0&0&0&0&0&0&0&0&0\\[3pt]
0&0&0&0&0&0&0&0&0\\[3pt]
0&0&0&0&0&0&0&0&0\\[3pt]
{\color{blue}0}&{\color{blue}0}&{\color{blue}$1\over16$}&{\color{blue}0}&{\color{blue}$\bf-{9\over16}$}&{\color{blue}1}&{\color{blue}$-{9\over16}$}&{\color{blue}0}&{\color{blue}$1\over16$}\\[3pt]
0&0&0&0&0&0&0&0&0\\[3pt]
0&0&0&0&0&0&0&0&0\\[3pt]
0&0&0&0&0&0&0&0&0\\[3pt]
0&0&0&0&0&0&0&0&0\\[3pt]
\end{tabular*}
}
\medskip
\subfigure[Wavelet filter associated with the mask $t_{0,1}(\ome_1,\ome_2)$]{
\small
\centering
\begin{tabular*}{0.35\textwidth}{@{\extracolsep{\fill}}ccccccccc}
0&0&0&0&{\color{blue}$1\over16$}&0&0&0&0\\[3pt]
0&0&0&0&{\color{blue}0}&0&0&0&0\\[3pt]
0&0&0&0&{\color{blue}$-{9\over16}$}&0&0&0&0\\[3pt]
0&0&0&0&{\color{blue}1}&0&0&0&0\\[3pt]
0&0&0&0&{\color{blue}$\bf-{9\over16}$}&0&0&0&0\\[3pt]
0&0&0&0&{\color{blue}0}&0&0&0&0\\[3pt]
0&0&0&0&{\color{blue}$1\over16$}&0&0&0&0\\[3pt]
0&0&0&0&{\color{blue}0}&0&0&0&0\\[3pt]
0&0&0&0&{\color{blue}0}&0&0&0&0\\[3pt]
\end{tabular*}
}
\medskip
\subfigure[Wavelet filter associated with the mask $t_{1,1}(\ome_1,\ome_2)$]{
\small
\centering
\begin{tabular*}{0.35\textwidth}{@{\extracolsep{\fill}}ccccccccc}
0&0&0&0&0&0&0&0&{\color{blue}${1\over16}$}\\[3pt]
0&0&0&0&0&0&0&{\color{blue}0}&0\\[3pt]
0&0&0&0&0&0&{\color{blue}$-{9\over16}$}&0&0\\[3pt]
0&0&0&0&0&{\color{blue}1}&0&0&0\\[3pt]
0&0&0&0&{\color{blue}$\bf-{9\over16}$}&0&0&0&0\\[3pt]
0&0&0&{\color{blue}0}&0&0&0&0&0\\[3pt]
0&0&{\color{blue}${1\over16}$}&0&0&0&0&0&0\\[3pt]
0&{\color{blue}0}&0&0&0&0&0&0&0\\[3pt]
{\color{blue}0}&0&0&0&0&0&0&0&0\\[3pt]
\end{tabular*}
}
\caption{Primal coset sum wavelet filters of Example 6 for 2-D with 4 vanishing moments ($\dm=k=2$)}
\label{figure:Waveletfilters}
\end{figure}

\medskip\noindent
{\bf Remark.}
We recall that orthogonality is a special case of biorthogonality. We note that the $\dm$-D Haar refinement mask (cf. Example 1 and 5) is orthogonal, whereas the $\dm$-D canonical coset sum Haar wavelet system (cf. Example 5) is not orthogonal. In fact, it is not possible to construct $\dm$-D canonical coset sum wavelet system that is orthogonal. This can be seen from the facts that the 1-D refinement mask we start with for such a wavelet system has to be interpolatory and orthogonal, and that there is no 1-D interpolatory orthogonal refinement mask (in the dyadic dilation) other than the Haar one (see, for example, \cite{JS}), whose associated $\dm$-D canonical coset sum wavelet system is not orthogonal as we just established. 
\endproof

\medskip
A drawback of the $\dm$-D coset sum Haar wavelet system in the previous example is that the wavelet masks have only one vanishing moment.
In order to construct $\dm$-D biorthogonal wavelet systems with larger number of vanishing moments, one needs to have $\dm$-D biorthogonal refinement masks with larger number of accuracy
(cf.~\S\ref{subS:masks}). In general, constructing $\dm$-D biorthogonal refinement masks with large number of accuracy can be cumbersome, especially when $\dm$ is large, since it involves solving a large number of linear equations. Since coset sum can preserve the biorthogonality and the accuracy number simultaneously, it allows one to bypass solving these linear systems to get biorthogonal refinement masks with large number of accuracy. Thus it is often easier to construct $\dm$-D wavelet systems based on the coset sum than other $\dm$-D wavelet systems, for large number of vanishing moments. The next is an example of such coset sum wavelet systems.

\medskip\noindent
{\bf Example 6: A family of $\dm$-D coset sum wavelet systems with larger number of vanishing moments.} We choose the univariate refinement masks $U_{2k}$ (interpolatory) and $S_{2k}$ as in (\ref{eq:defU2k}) and (\ref{eq:defS2k}), respectively, and apply Theorem~\ref{thm:cosetsumwavelet}. Then with the primal wavelet
masks given as
$$
t_{\nu}(\ome)=e^{-i\ome\cdot\nu}\sin^{2k}({\ome\cdot\nu\over
2})P_k(\cos^2({\ome\cdot\nu\over 2})),\quad \nu\in\set',
$$
with $\set'=\{0,1\}^\dm\bks0$,
there exist dual wavelet masks $t_{\nu}^\du$, $\nu\in\set'$, such that $(\mathcal{C}_n[S_{2k}],(t_{\nu})_{\nu\in\set'})$ and
$(\mathcal{C}_n[U_{2k}],(t_{\nu}^\du)_{\nu\in\set'})$ are $\dm$-D combined biorthogonal masks.
It is easy to check that each of the wavelet masks of this wavelet system has $2k$ vanishing moments. All the primal wavelet filters are supported on the union of $2^\dm-1$ line segments along $\nu$ direction for $\nu\in\set'$. For example, if $\dm=2$, then $\set'=\{(1,0), (0,1), (1,1)\}$ and the primal wavelet masks for the case $k=2$ are given as
\beaN
t_{(1,0)}(\ome_1,\ome_2)&\,=\,&e^{-i\ome_1}\sin^4({\ome_1\over 2})\left(1+2\cos^2({\ome_1\over 2})\right),\\
t_{(0,1)}(\ome_1,\ome_2)&=&e^{-i\ome_2}\sin^4({\ome_2\over 2})\left(1+2\cos^2({\ome_2\over 2})\right),\\
t_{(1,1)}(\ome_1,\ome_2)&=&e^{-i(\ome_1+\ome_2)}\sin^4({\ome_1+\ome_2\over
2})\\
&\cdot&\left(1+2\cos^2({\ome_1+\ome_2\over 2})\right).
\eeaN
The associated wavelet filters are depicted in Figure~\ref{figure:Waveletfilters}. The magnitude of the primal masks $\mathcal{C}_2[S_4],t_{(1,0)},t_{(0,1)}$, and $t_{(1,1)}$ (i.e. the magnitude of the frequency responses of the filters associated with the primal masks) are depicted in Figure~\ref{figure:cosetsum_fr}. The magnitude of the corresponding tensor product primal masks are given in Figure~\ref{figure:tensor_fr} for comparison.

Since $\mathcal{C}_2[U_{4}]$ and $\mathcal{C}_2[S_{4}]$ generate the refinable functions that are in $L^2(\RR^2)$ (cf.~Example 4 and Figure~\ref{figure:s4u4refinement}), and since we have FIR filters, the $2$-D coset sum wavelet system generated from the combined biorthogonal masks $(\mathcal{C}_2[S_4],t_{(1,0)},t_{(0,1)},t_{(1,1)})$ and $(\mathcal{C}_2[U_4],t^\du_{(1,0)},t^\du_{(0,1)},t^\du_{(1,1)})$ is also in $L^2(\RR^2)$. \endproof

\begin{figure}[t]
\hspace{-0.7cm}
\includegraphics[scale=0.5]{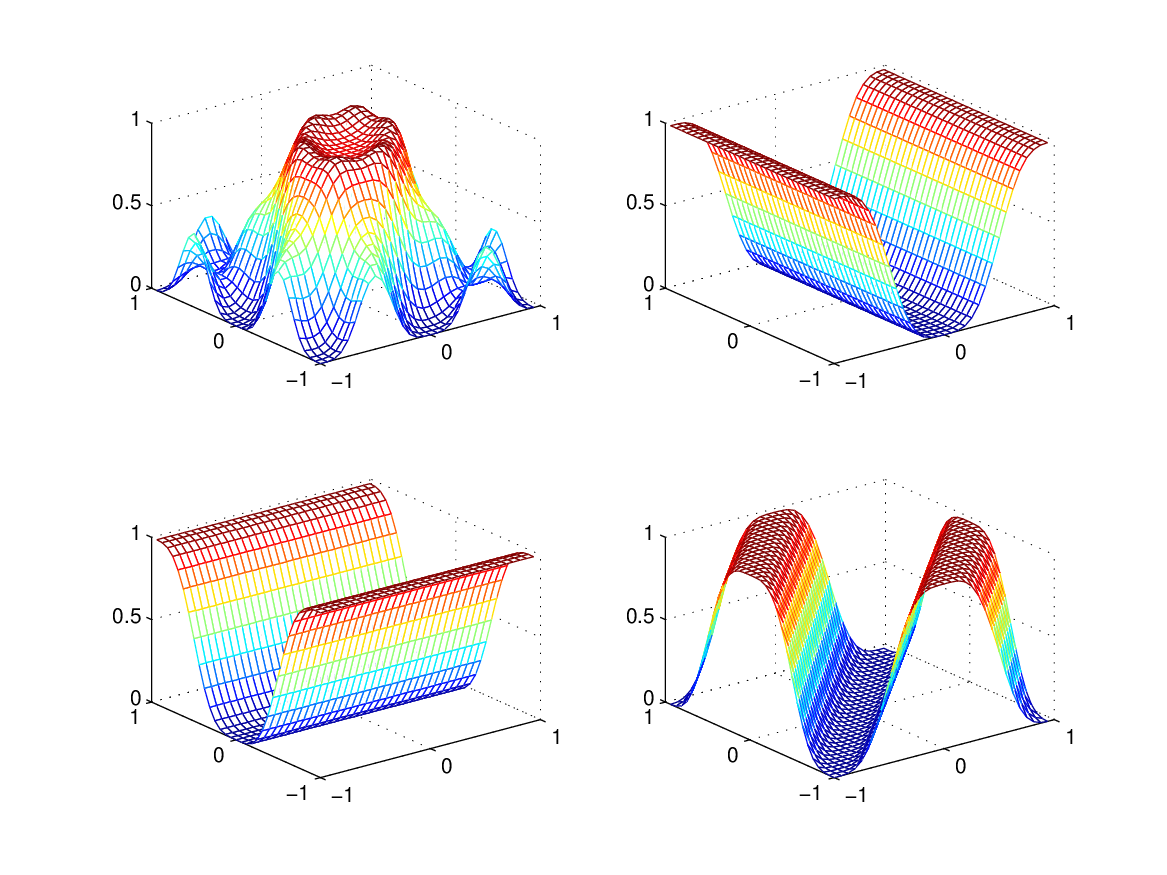}
\caption{The magnitude of the primal coset sum masks $\mathcal{C}_2[S_4],t_{(1,0)},t_{(0,1)}$, and $t_{(1,1)}$ in Example 6}
\label{figure:cosetsum_fr}
\end{figure}

\begin{figure}[t]
\hspace{-0.7cm}
\includegraphics[scale=0.5]{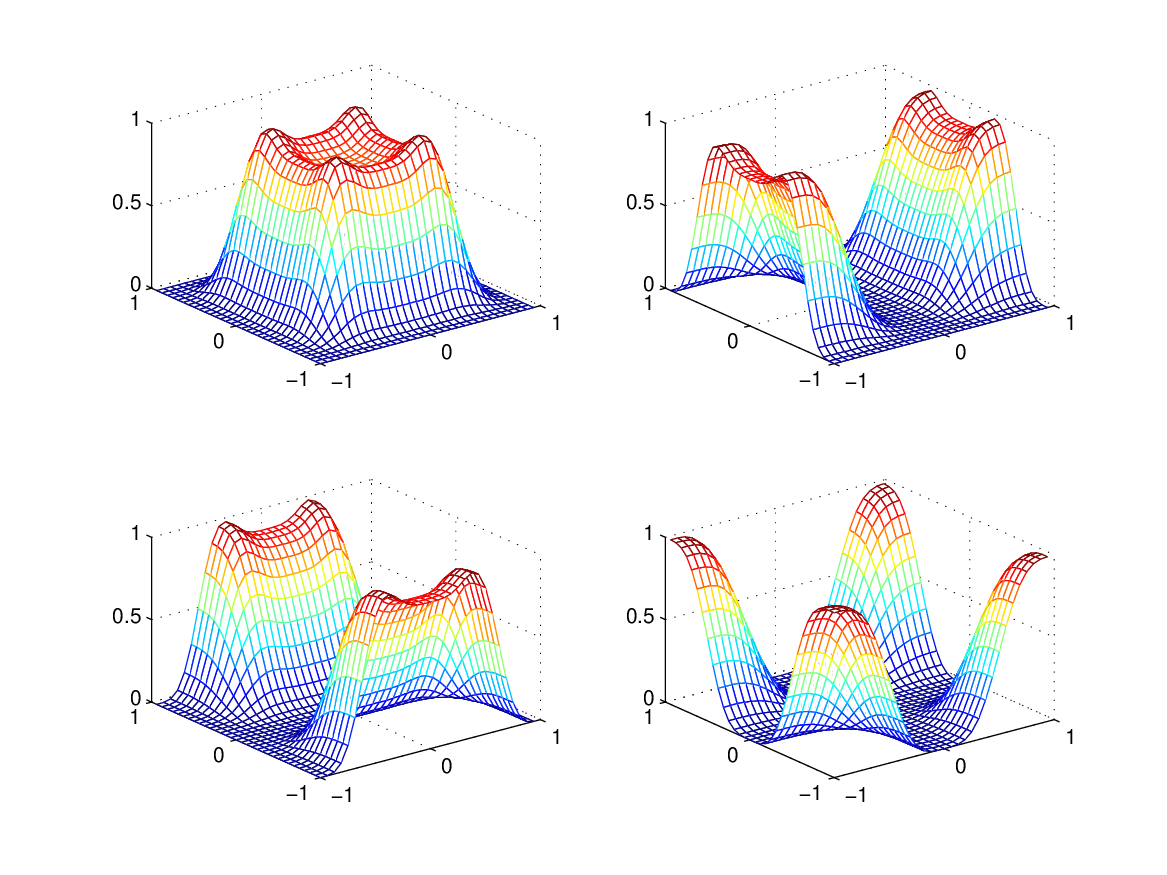}
\caption{The magnitude of the primal tensor product masks that are comparable to the masks in Figure~\ref{figure:cosetsum_fr}}
\label{figure:tensor_fr}
\end{figure}

\medskip
It should be noted that the above properties (the space localization property discussed in Example~5, and the frequency responses, the vanishing moments, and the smoothness--in the sense of whether or not a wavelet system belongs to $L^2$--discussed in Example 6) of canonical coset sum wavelet systems may not hold true for other coset sum wavelet systems.

\subsection{Fast coset sum wavelet algorithms}
\label{subS:Algorithms}
Next we show that the canonical coset sum wavelet system can be associated with the fast algorithm with linear complexity whose complexity constant does not grow with the spatial dimension. When presenting and analyzing our algorithm below, we use mostly filters instead of masks that we have been used so far, as this approach will be more useful in practice. 

\medskip\noindent{\bf Fast Coset Sum Wavelet Algorithms.}
Let $S$ and $U$ be biorthogonal univariate refinement masks, where $U$ is interpolatory. Let $G$ and $H$ be the filters associated with the refinement masks $S$ and $U$, respectively. In particular, $H$ is interpolatory (cf.~(\ref{eq:interpfilter})).
%Then $h$ and $H$ are real-valued filters.

\medskip
{\obeylines{{\tt
{\it input}  $y_J: \Zd\to\RR$
\smallskip
{\bf (1) Decomposition Algorithm:}
$a_G=-2^\dm+2+(2^\dm-1)G(0)$
for $j=J,J-1,\cdots,1$
\ for $k\in\Zd$
\ \ $y_{j-1}(k)$
\ \ \ $=\disp{1\over 2^\dm}(a_Gy_j(2k)+\disp\sum_{\nu\in \set'}\disp\sum_{L\in\ZZ\bks 0} G(L)y_j(2k+L\nu))$ (i)
\ end
\ for $\nu\in\set'$ and $k\in\Zd$
\ \ $w_{\nu,j-1}(k)$
\ \ \ $=\disp{1\over 2}(y_j(2k+\nu)-\disp\sum_{m\equiv 1}H(m)y_{j}(2k+(1-m)\nu))$ (ii)
\ end
end

\medskip
{\bf (2) Reconstruction Algorithm:}
for $j=1,\cdots,J-1,J$
\ for $k\in\Zd$
\ \ $y_j(2k)$
\ \ \ $=y_{j-1}(k)-{\disp1\over \disp2^{\dm-1}}\disp\sum_{\nu\in\set'}\disp\sum_{L\in\ZZ}G(2L+1)w_{\nu,j-1}(k+L\nu)$ 
\ \ \ \ \ \ \ \ \ \ \ \ \ \ \ \ \ \ \ \ \ \ \ \ \ \ \ \ \ \ \ \ \ \ $\:$ (iii)
\ end
\ for $\nu\in\set'$ and $k\in\Zd$
\ \ $y_j(2k+\nu)$
\ \ \ $=2w_{\nu,j-1}(k)+\disp\sum_{m\equiv 1}H(m)y_{j}(2k+(1-m)\nu)$ (iv)
\ end
end
}}}%\eject

Given coarse coefficients $y_j$ at level $j$, Decomposition Algorithm first computes the lower level coarse coefficients $y_{j-1}$, and then the wavelet coefficients $w_{\nu,j-1}$, $\nu\in\set'=\{0,1\}^\dm\bks0$. The coefficients $y_{j-1}$ and $w_{\nu,j-1}$ are obtained by filtering (using the $\dm$-D filter $g$ associated with the coset sum refinement mask $\mathcal{C}_\dm[S]$ for $y_{j-1}$ and the $\dm$-D filter $h_\nu$ associated with the primal coset sum wavelet mask $t_\nu$ for $w_{\nu,j-1}$) followed by downsampling, as is typically done in wavelet decomposition process (see, for example, \cite{Da}). Since the $\dm$-D mask $\mathcal{C}_\dm[S]$ can be written in terms of $1$-D mask $S$ (cf. Definition~\ref{def:cosetsum}), the associated $\dm$-D filter $g$ can be written in terms of $1$-D filter $G$ (cf. (\ref{eq:filter})). Similarly, from the fact that the $\dm$-D mask $t_\nu$ can be written in terms of $1$-D mask $U$ (cf. (\ref{eq:nDwaveletmask})), the associated $\dm$-D filter $h_\nu$ can be written in terms of $1$-D filter $H$. Taking into account these observations, we get the above simple expressions for $y_{j-1}$ and $w_{\nu,j-1}$ in Step {\tt (i)} and  {\tt (ii)}, respectively. In Step {\tt (ii)} and {\tt (iv)}, $m\equiv 1$ is used to mean that $m$ is congruent to $1$ in modulo $2$, i.e., $m$ is an odd integer. 

Reconstruction Algorithm recovers $y_j$ from $y_{j-1}$, and $w_{\nu,j-1}$, $\nu\in\set'$. It first recovers $y_j$ at even points (cf.~Step {\tt (iii)}) and then at all other points (cf.~Step {\tt (iv)}). Step {\tt (iii)} is a key step in making our algorithm fast (cf.~Complexity discussion below). It is easy to show that the identity in Step {\tt (iii)} holds true for our canonical coset sum wavelet system (see Appendix~\ref{subS:stepiii} for proof), but it need not be true for other coset sum wavelet systems. Step {\tt (iv)} is simply a reverse process of Step {\tt (ii)} and is possible since the only $y_j$ values we need at this step are the values at even points, and these are already computed in Step {\tt (iii)}.

For a given pair of $1$-D masks $S$ and $U$ that generates the canonical coset sum wavelet system, the filters $H$ and $G$ in the algorithm can be computed easily. For example, for the $\dm$-D coset sum Haar wavelet system in Example~5, both $H$ and $G$ are the $1$-D Haar refinement filter (cf. Example 1). For the coset sum wavelet system in Example 6 that is generated from $S_{4}$ (cf.~(\ref{eq:defS2k})) and $U_{4}$ (cf.~(\ref{eq:defU2k})), the filters are given as 
$$H(K)=\cases{1,&$K=0$,\cr
	{9\over 16},&$K=\pm1$,\cr
	-{1\over 16},&$K=\pm3$,\cr
     0,&otherwise,\cr}
G(K)=\cases{{696\over512},&$K=0$,\cr
	{288\over512},&$K=\pm1$,\cr
	-{126\over512},&$K=\pm2$,\cr
	-{32\over512},&$K=\pm3$,\cr
	{36\over512},&$K=\pm4$,\cr
	-{2\over512},&$K=\pm6$,\cr
     0,&otherwise.}$$

We note that the above algorithms for the canonical coset sum wavelet systems are not redundant: the number of coefficients after the decomposition algorithm is approximately the same as the number of input samples, assuming that the filter length of each filter involved in the algorithm is negligible compared to the number of input samples. 

\medskip\noindent{\bf Complexity.}
We measure complexity by counting the number of operations needed in order
to fully derive $y_{j-1}$, and $w_{\nu,j-1}$, $\nu\in\set'$, from $y_j$, and add the number of
operations needed for the reconstruction. Here, we count only multiplicative operations such as multiplication and division, as counting additive operations gives a similar result. 

As in the fast tensor product wavelet algorithms discussed in \S\ref{subS:tensorproduct}, the complexity here is linear, i.e. $\sim
CN$, with
$N$ the number of nonzero entries in $y_J$, and $C$ some constant independent of $y_J$. We refer to this constant as the {\it constant in the complexity bound} or simply as the {\it complexity constant} throughout this paper. 

We now estimate the complexity constant for fast coset sum wavelet algorithms by computing the mean number of operations per single entry in $y_J$. Suppose $\alpha$ and $\beta$ are the numbers of nonzero entries of the filters $G$ and $H$, respectively. Then, the number of operations that are needed to process the portion of $y_J$ that lies on the vertices of a unit cube is the sum of 
\begin{itemize}
\item $2\dm-1$ (for computing $a_G$),   
\item $(2^\dm-1)(\alpha-1)+\dm+1$ (for Step {\tt (i)}),
\item $2(2^\dm-1)\beta$ (for Step {\tt (ii)} and {\tt (iv)}), and
\item $(2^\dm-1){{\alpha+1}\over 2}+\dm-1$ (for Step {\tt (iii)}).
\end{itemize} 
After computing the sum, we divide it by $2^\dm$, which is the number of vertices in the unit cube, in order to obtain the cost per entry of performing one complete cycle of decomposition/reconstruction. As a result, we get 
$${3\over 2}\alpha + 2\beta$$ 
as an upper bound for the cost per entry. 
Therefore, the algorithm has complexity $({3\over 2}\alpha + 2\beta)N$, and the constant in the complexity bound in this case is ${3\over 2}\alpha + 2\beta$, which does not increase as the spatial dimension $\dm$ increases. A similar argument is used in \cite{HR3} to compute the complexity constant for the algorithm introduced there. 

Contrary to the complexity constant of the fast coset sum wavelet algorithm that we just computed, in the tensor product case the constant grows with the dimension (cf. \S\ref{subS:tensorproduct}). There are a couple of components that make the coset sum wavelet algorithm this fast. First, as we discussed in \S\ref{subS:cosetsumwaveletsystems} (cf.~(\ref{eq:nDwaveletmask})), the wavelet masks of the coset sum wavelet system are essentially univariate. Second, as we can see from the above algorithms (cf.~Step~{\tt (iii)}), the reconstruction step can be done by completely bypassing the dual wavelet filters. This is reminiscent of the Laplacian pyramid \cite{BA} (cf.~Appendix~\ref{subS:proofofThm2}) and its variant \cite{Hur}, which have trivial reconstruction steps that are simply reverse processes of decomposition steps. As a consequence, our algorithm inherits an asymmetry in the roles of the lowpass filters from the Laplacian pyramid. Hence the 1-D lowpass filters $G$ and $H$ in our fast coset sum wavelet algorithm play different roles.
\endproof

\medskip\noindent
{\bf Remark 1.}
It is well known that any (MRA-based) biorthogonal wavelet system (associated with FIR filters) has decomposition and reconstruction algorithms with linear complexity (see, for example, \cite{Ma1,Ma2,CDF,SN}). In fact, as we alluded to earlier, our fast coset sum decomposition algorithm is nothing but this generic decomposition algorithm for the given canonical coset sum wavelet system. However, our fast coset sum reconstruction algorithm is fundamentally different from this generic reconstruction algorithm: the dual coset sum wavelet filters (cf.~(\ref{eq:tnudu})) that are not used for our reconstruction algorithm are used for the generic one. As a result, our canonical coset sum wavelet system in Theorem~\ref{thm:cosetsumwavelet} has two different algorithms (the fast coset sum wavelet algorithm and the generic one) and the generic algorithm is always slower than the fast coset sum wavelet algorithm.
\endproof

\medskip\noindent
{\bf Remark 2.}
For any biorthogonal wavelet system, multiplying the primal part with some constant factors and dividing the dual part with the same factors will still make a biorthogonal wavelet system. As the functions in these two systems differ only by constants, it is clear that the two systems are essentially the same and most of their properties--including the support and the smoothness--are kept the same.

For the above fast coset sum wavelet algorithms, this means that the decomposition step can be rewritten with explicit {\it normalization factors} $c, d > 0$ as
$$y_{j-1}^{new}(k)=cy_{j-1}(k),\quad w_{\nu,j-1}^{new}(k)=dw_{\nu,j-1}(k)$$
where $y_{j-1}(k)$, $w_{\nu,j-1}(k)$ are defined as in Step {\tt (i)}-{\tt (ii)}, and that the reconstruction step can be modified accordingly: the expressions in the right-hand side of Step {\tt (iii)}-{\tt (iv)} can be rewritten in terms of  
$${1\over c}y_{j-1}^{new}(k),\quad {1\over d}w_{\nu,j-1}^{new}(k)$$
in place of $y_{j-1}(k)$, $w_{\nu,j-1}(k)$ that are currently used. In this sense, our original fast coset sum wavelet algorithms can be considered as a special case when $c=d=1$. These normalizations are used throughout this paper except in~\S\ref{subS:Experiments}  (see the discussion below and the footnote in the subsection).

When normalization factors are used for the wavelet algorithms, most properties of the algorithms are not affected. For example, fast coset sum wavelet algorithms with normalization factors will still have the linear complexity with the complexity constant that is independent of $\dm$. 
However the use of different normalization factors may result in different performance in practice \cite{G}. For example, when the algorithms are used for nonlinear approximation with multiple levels (cf.~\S\ref{subS:Experiments}), the coefficients are multiplied by constant factors and these factors propagate recursively to other coefficients in lower levels and, as a result, the use of normalization factors may change the relative size of the coefficients.
\endproof

\medskip
Below we compare the fast tensor product wavelet algorithms with the fast coset sum wavelet algorithms, both based on the Deslauriers-Dubuc mask and its dual mask in \S\ref{subS:cosetsumproperties}.

\medskip\noindent
{\bf Example 7: Fast tensor product wavelet algorithms vs. fast coset sum wavelet algorithms.}
In this example, we compare the algorithms for two different families of $\dm$-D wavelet systems constructed from the same univariate refinement masks by using two different methods: (I) the tensor product and (II) the coset sum. We consider the same univariate refinement masks as in Example 4 and 6, i.e. $U_{2k}$ (interpolatory) and $S_{2k}$ as in (\ref{eq:defU2k}) and (\ref{eq:defS2k}), respectively. 
It is easy to see that the number of nonzero entries of the filter associated with $S_{2k}$ is $\alp=8k-3$, and the number of nonzero entries of the filter associated with $U_{2k}$ is $\beta=2k+1$.

Then complexity constant for each algorithm is given as follows:
\begin{enumerate}[(I)]
\item (Tensor Product Case) From \S\ref{subS:tensorproduct}, the complexity constant for the fast tensor product algorithm is $(\alp+\beta)n=(10k-2)n$, which grows linearly with the dimension.
\item (Coset Sum Case) From the above Complexity discussion, the complexity constant for the fast coset sum wavelet algorithm is ${3\over 2}\alpha + 2\beta={3\over 2}(8k-3)+2(2k+1)=16k-{5\over 2}$, which {\it does not} grow with the dimension.
\end{enumerate}
Therefore, remarkably, if we fix $k$ (hence the number of vanishing moments of the wavelet system) and increase the dimension $\dm$, then the {\it complexity constant stays the same for the coset sum case, whereas it increases for the tensor product case}.
%More precisely, the constant is proportional to $k\dm$.
\endproof

\begin{figure}[t]
\centering
\subfigure[Original image ``part of lena"]{
\includegraphics[width=0.22\textwidth]{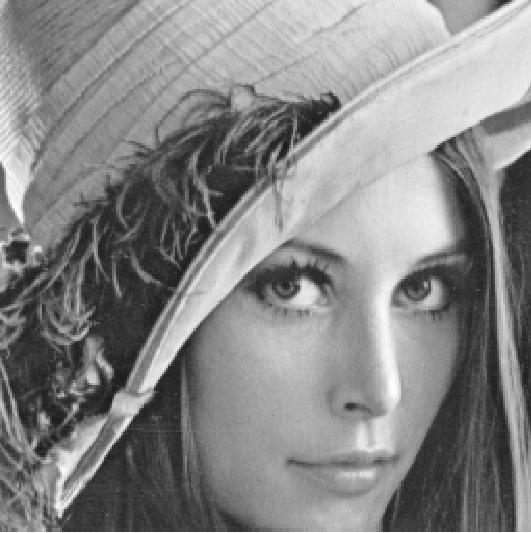}
}~~~
\subfigure[PSNR of reconstructed image]{
\includegraphics[width=0.23\textwidth]{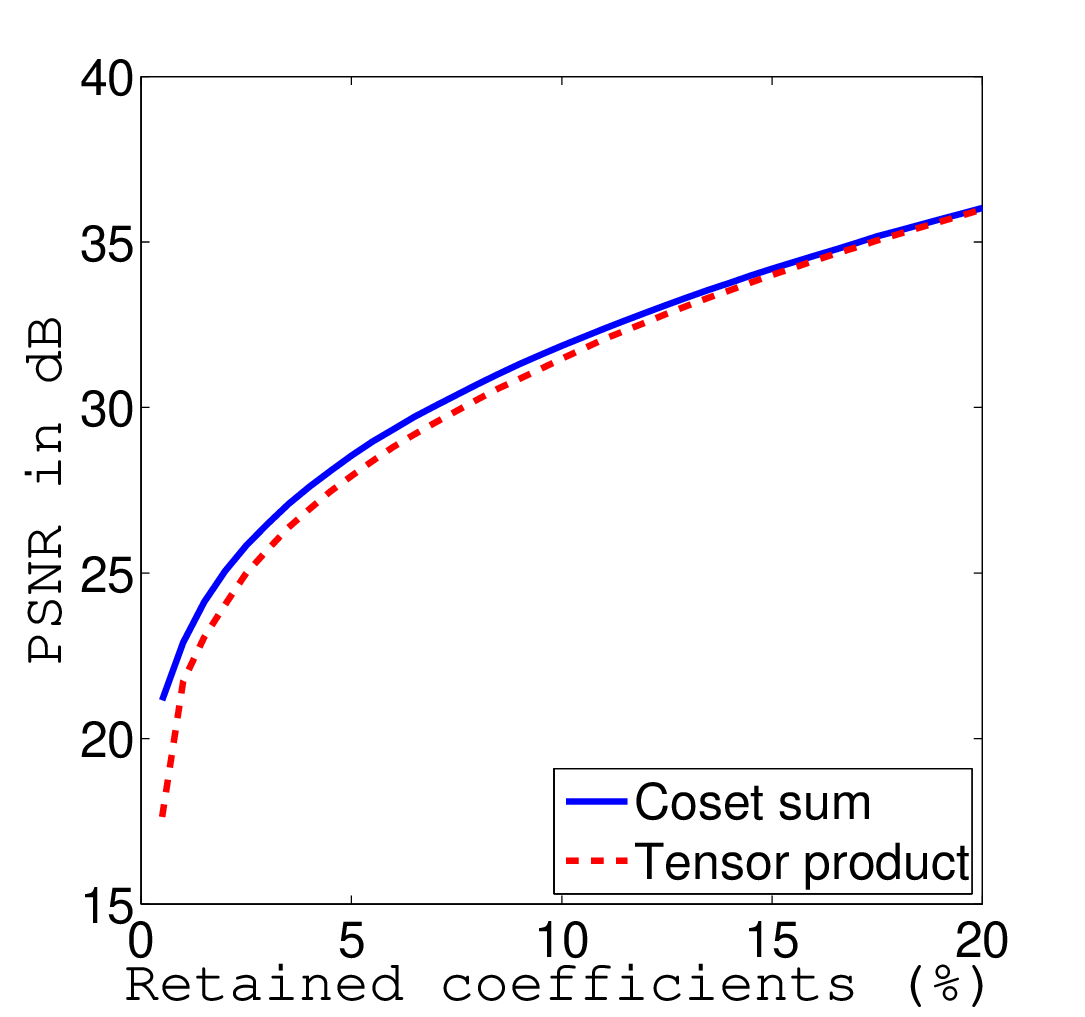}
}
\subfigure[Reconstructed by tensor product]{
\includegraphics[width=0.22\textwidth]{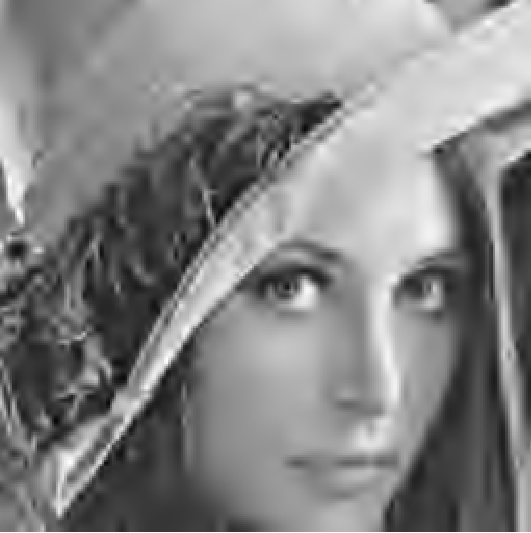}
}~~~
\subfigure[Reconstructed by coset sum]{
\includegraphics[width=0.22\textwidth]{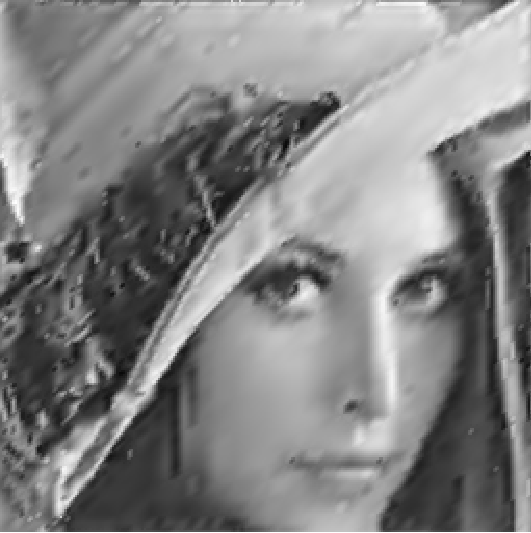}
}
\caption{Comparison of approximation power of tensor product and coset sum for (a) original image ``part of lena": $5$-level-down decomposition and reconstruction using $3\%$ largest coefficients. (c) The reconstructed image by tensor product, PSNR $=25.7$ dB. (d) The reconstructed image by coset sum with $\set'=\{(1,0),(0,1),(1,1)\}$, showing sharper edges and better visual quality, with improved PSNR $=26.5$ dB. (b) PSNR of reconstructed images over different percentage of retained coefficients ($0.5\%$-$20\%$). This experiment shows that the reconstructed images by coset sum have higher PSNR (solid blue), hence better approximation quality than those by tensor product (dotted red) over the range $0.5\%$-$20\%$ for image ``part of lena".}
\label{figure:lena}
\end{figure}

\begin{figure}[t]
\centering
\subfigure[Original image ``wood45"]{
\includegraphics[width=0.22\textwidth]{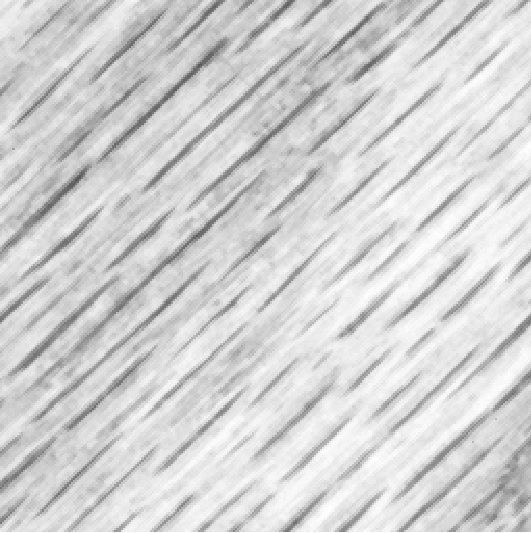}
}~~~
\subfigure[PSNR of reconstructed image]{
\includegraphics[width=0.23\textwidth]{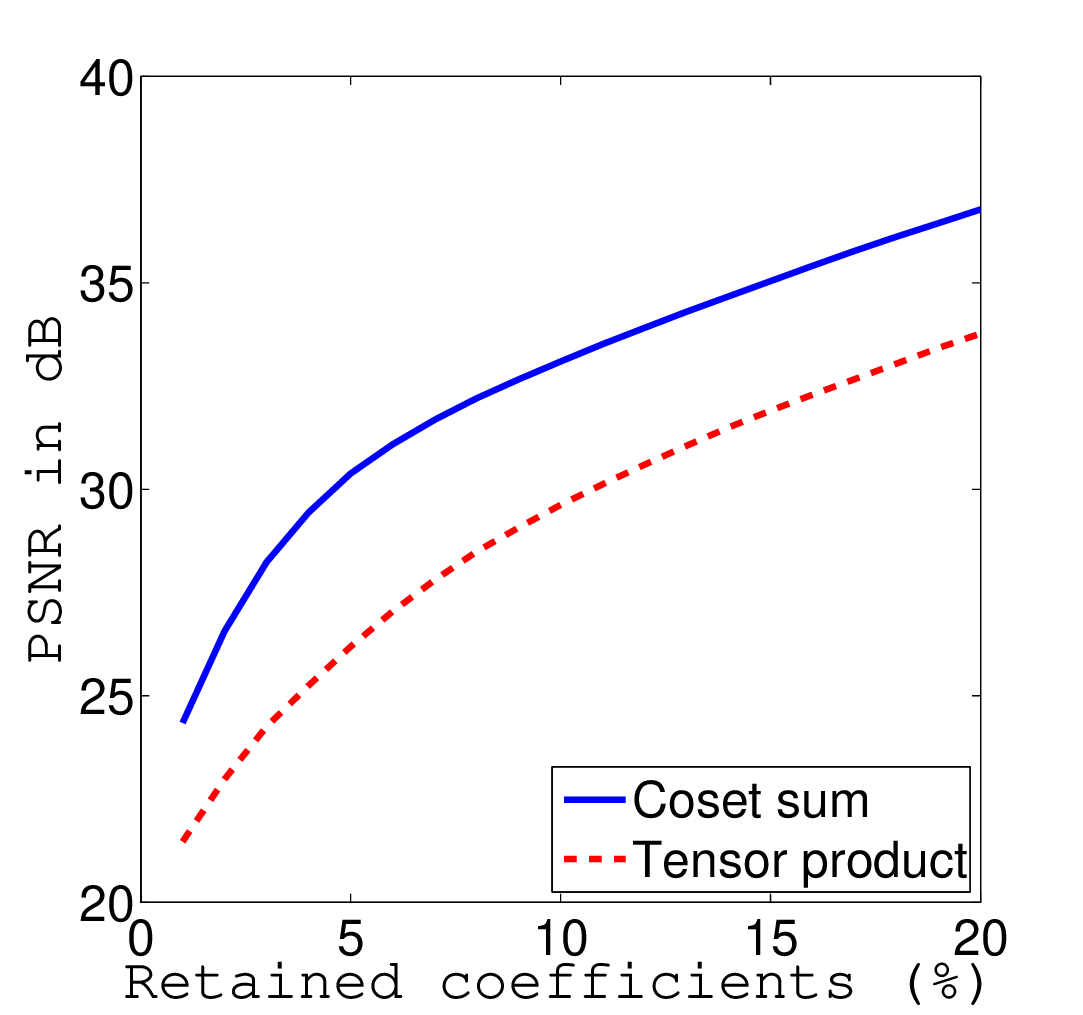}
}
\subfigure[Reconstructed by tensor product]{
\includegraphics[width=0.22\textwidth]{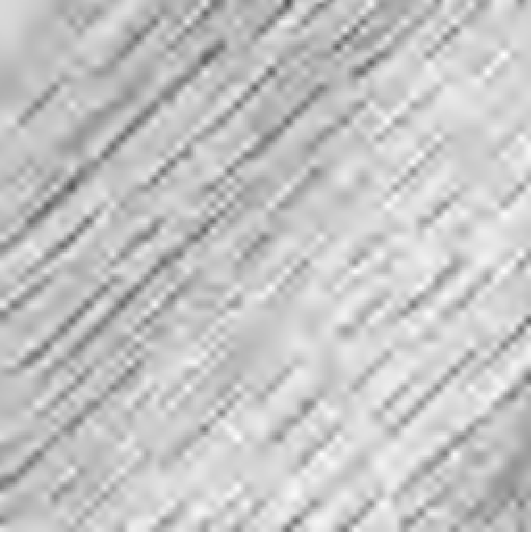}
}~~~
\subfigure[Reconstructed by coset sum]{
\includegraphics[width=0.22\textwidth]{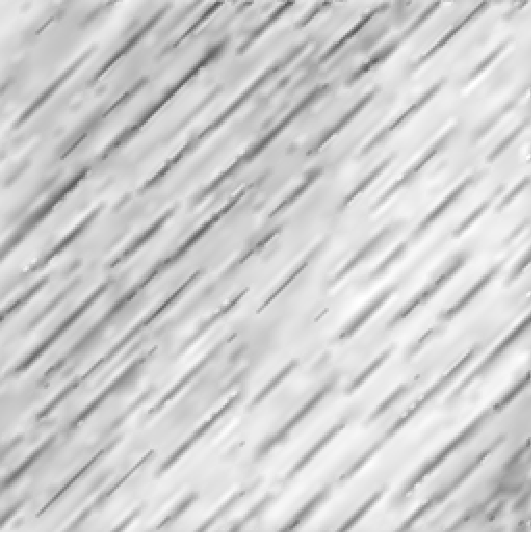}
}
\caption{Comparison of approximation power of tensor product and coset sum for (a) original image ``wood45": $5$-level-down decomposition and reconstruction using $3\%$ largest coefficients. (c) The reconstructed image by tensor product, with blurry recovered content and PSNR $=24.3$ dB. (d) The reconstructed image by coset sum with $\set'=\{(1,0),(0,1),(1,1)\}$, showing better approximation to the original image and better improved PSNR $=28.2$ dB. (b) PSNR of reconstructed images over different percentage of retained coefficients ($0.5\%$-$20\%$). The improvement of PSNR in this example is larger than that in ``part of lena" example due to the stronger directional content in image ``wood45".}
\label{figure:wood45}
\end{figure}

\begin{figure}[t]
\centering
\subfigure[Original image: ``wood"]{
\includegraphics[width=0.22\textwidth]{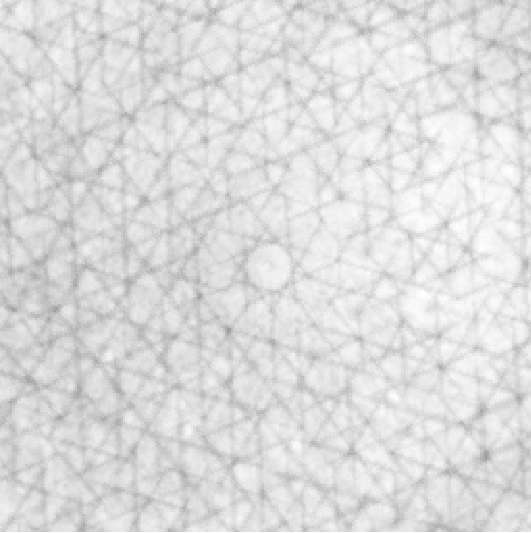}
}~~~
\subfigure[PSNR of reconstructed image]{
\includegraphics[width=0.23\textwidth]{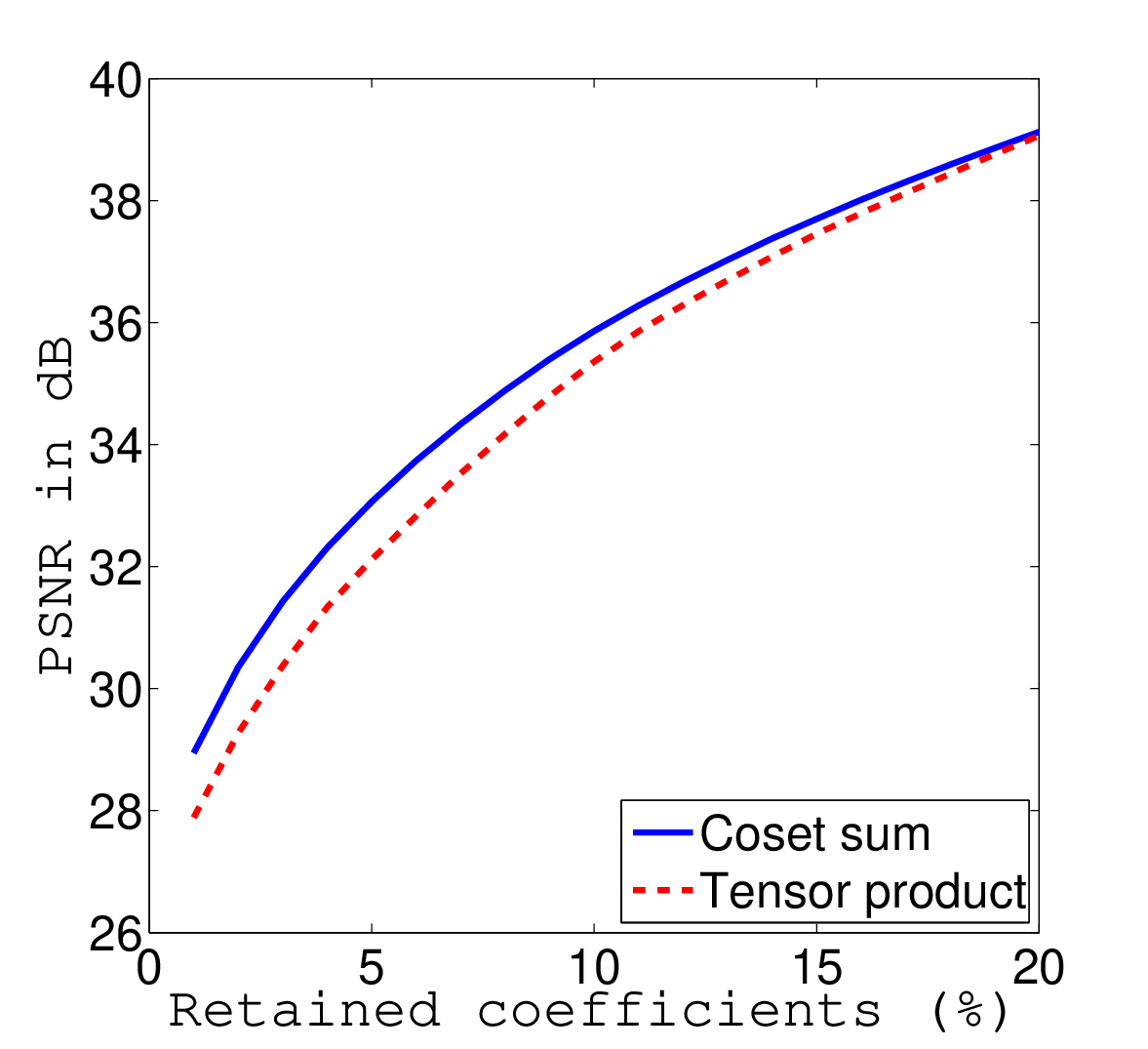}
}
\subfigure[Reconstructed by tensor product]{
\includegraphics[width=0.22\textwidth]{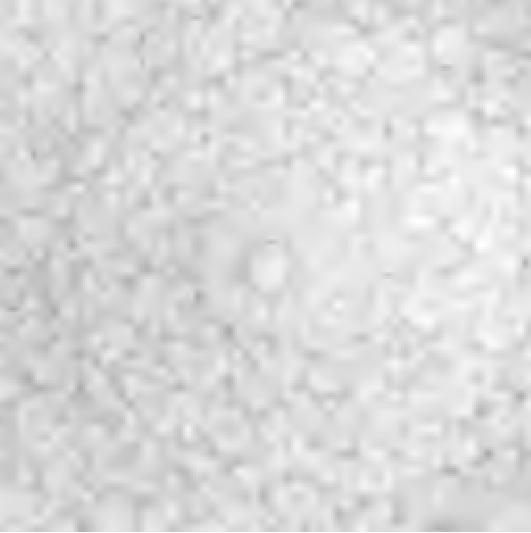}
}~~~
\subfigure[Reconstructed by coset sum]{
\includegraphics[width=0.22\textwidth]{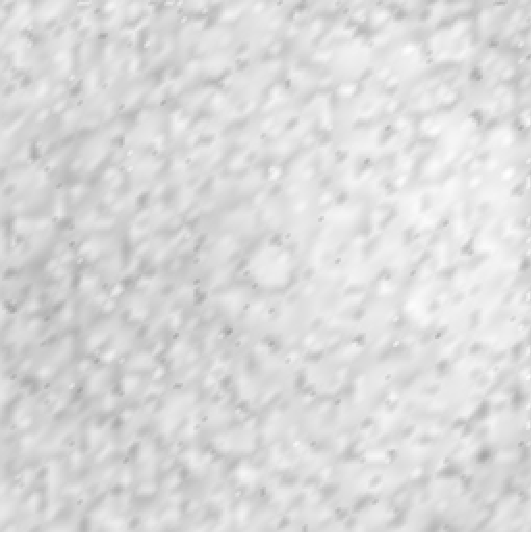}
}
\caption{An example of image with multiple directions. Comparison of approximation power of tensor product and coset sum for (a) original image ``wood": $5$-level-down decomposition and reconstruction using $3\%$ largest coefficients. (c) The reconstructed image by tensor product, PSNR $=30.4$ dB. (d) The reconstructed image by coset sum with $\set'=\{(1,0),(0,1),(1,1)\}$, PSNR $=31.4$ dB. (b) PSNR of reconstructed images over different percentage of retained coefficients ($0.5\%$-$20\%$).}
\label{figure:wood}
\end{figure}

\subsection{Experiments}
\label{subS:Experiments}
In this subsection we present some experimental results of the canonical coset sum wavelet system, in comparison with the tensor product wavelet system. We have implemented the fast coset sum wavelet algorithms in Matlab. The program takes a pair of $1$-D biorthogonal refinement filters as input and works for $2$-D images. We compare our Matlab program with the standard Matlab implementation of $2$-D fast tensor product wavelet algorithms: {\it wavedec2} (for decomposition) and {\it waverec2} (for reconstruction) in Wavelet Toolbox \cite{WaveletToolbox}. \footnote{When comparing the implementation of two different wavelet systems, it is important to use the same normalization factors as they may affect the performance (cf. Remark 2 after Complexity discussion). Normalization factors $c=d=2$ are used for implementing our coset sum wavelet system since these are the normalization factors used for the tensor product Matlab implementation when seen in terms of a $2$-D generalization of the related 1-D concepts (i.e. the DC and Nyquist gains) \cite{RJJ}.}{For} the experiments in this subsection, we use two different $2$-D wavelet systems obtained from the same $1$-D filters, $U_4$ (shown in Figure \ref{figure:C2U4}) and $S_4$ (shown in Figure \ref{figure:C2S4}), but using two different methods, coset sum and tensor product. For the coset sum wavelet system, we initially choose $\set'=\{(1,0),(0,1),(1,1)\}$ as the nonzero coset representatives. The two wavelet systems constructed this way are discussed in Example 6, and the complexity constants of their algorithms are compared in Example 7. 

We first compare the running time of the fast coset sum wavelet algorithm with the fast tensor product wavelet algorithm. We apply the two wavelet systems constructed as above to test images, ``part of lena''\footnote{
This image is obtained from the image ``lena" ($512 \times 512$) in the image repository http://links.uwaterloo.ca/Repository.html by taking its central part (of size $256 \times 256$).
%In this paper, we only used part of the original image lena, which is widely used in image processing literature.
} in Figure \ref{figure:lena}(a), and ``wood45''\footnote{
This image is obtained from the image ``wood.000" ($512 \times 512$) in the SIPI Image Database http://sipi.usc.edu/database/database.php?volume=rotate by rotating $45^{\circ}$ clockwise, and taking its central part (of size $256 \times 256$).
} in Figure \ref{figure:wood45}(a), both of which have directional content along the diagonal direction. Here, the diagonal direction, or $45^{\circ}$ from the positive $x$-axis, is chosen because it can highlight the benefit of our coset sum wavelet system over the tensor product wavelet system: it is one of the directions that may be captured well by our coset sum system since $\tan 45^{\circ}={1\over 1}$ and $(1,1)\in\set'$, while it is one of the directions that may not be captured well by the tensor product system since it is not a coordinate direction. Both are of size $256 \times 256$, and we perform $5$-level-down decomposition and reconstruction. The running time for ``part of lena'' is about $0.0279$ seconds~(s) on average for tensor product algorithm and about $0.0161$ s on average for coset sum algorithm, on a Mac 4G 1333MHz laptop. The running time for ``wood45" is about $0.0283$ s on average for tensor product and about $0.0162$ s on average for coset sum. We also tried several other images, both with and without directional content, for various levels of decomposition and reconstruction, and obtained essentially the same results: the coset sum algorithms were faster than the tensor product ones. These experiments confirm our theoretical finding in the previous subsection (cf. Example 7).

\begin{figure*}[t]
\centering
\subfigure[Reconstructed ``part of lena"]{
\includegraphics[width=0.22\textwidth]{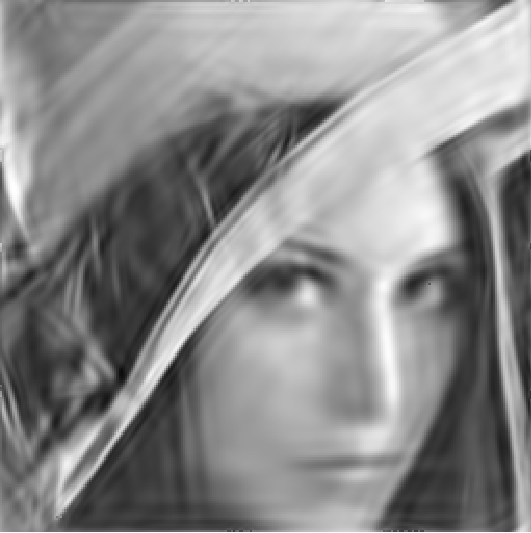}
}~~~
\subfigure[Reconstructed ``wood45"]{
\includegraphics[width=0.22\textwidth]{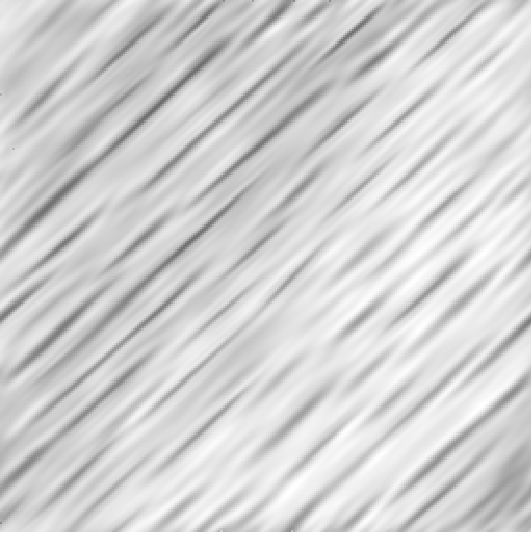}
}~~~
\subfigure[Reconstructed ``wood"]{
\includegraphics[width=0.22\textwidth]{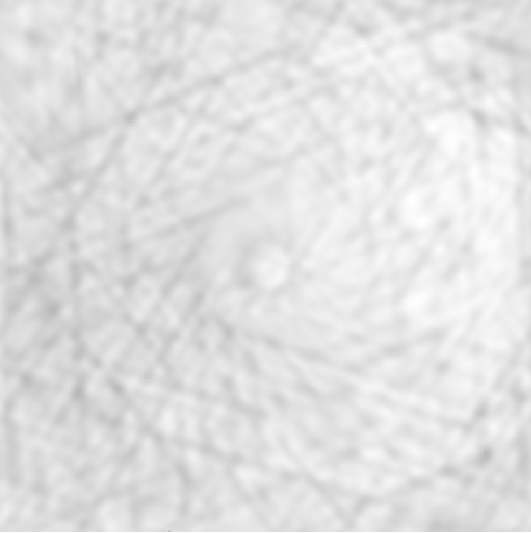}
}
%\subfigure[``wood" reconstructed 2]{
%\includegraphics[width=0.22\textwidth]{best_curvelet(0,1,5,32).png}
%}
\caption{Reconstructed images by curvelet from a $5$-level-down decomposition and retaining $3\%$ largest coefficients, for original images ``part of lena" (Figure 11(a)), ``wood45" (Figure 12(a)) and ``wood" (Figure 13(a)). Experiments are done by using the wrapping-based method (\cite{CDDY}) implemented in CurveLab Toolbox (\cite{CurveLabToolbox}) with the default parameter setting. PSNR for reconstructed images are: (a) $25.4$ dB (b) $27.7$ dB and (c) $30.7$ dB. }
\label{figure:curvelet}
\end{figure*}

We now compare the approximation power of these two wavelet systems. For this, we first decompose a fixed image using the two wavelet systems, then recover the image from the $M$-largest decomposed coefficients
(in magnitude), and finally compare the Peak-Signal-to-Noise-Ratio (PSNR) of the two reconstructed images. The reconstructed image with higher PSNR indicates better approximation to the original image. 
For the image ``part of lena'', the reconstructed image using coset sum system shows sharper edges along the diagonal direction and better visual quality than those of tensor product system, and has a slightly higher PSNR
%and better quality than that of tensor product system, and also has higher PSNR 
(Figure~\ref{figure:lena}(a)(c)(d)). In this experiment, we found that as long as the percentage of retained coefficients is not too large, the coset sum system has slightly higher PSNR (see Figure \ref{figure:lena}(b) for the range 0.5\%-20\%). For higher percentage, the coset sum system showed either comparable or slightly worse performance.
Another example using the texture image ``wood45" is also presented (Figure~\ref{figure:wood45}). The reconstructed image using coset sum system shows even better performance in this example in terms of PSNR, which is probably due to its stronger directional content. Contrary to the previous experiment with ``part of lena'', in this experiment, the coset sum system showed consistently higher PSNR for all the percentages. From this experiment, we see that coset sum wavelet system shows promising results when applied to images with strong directional content that matches with the directions of the coset sum primal wavelet filters.
%%% Rtype:R1212-2

We recall that the directional preference of the coset sum primal wavelet filters can be specified by the associated coset representatives in $\set'$. If a dominant direction of the given image does not match with the preferred directions of the coset sum, the coset sum wavelet system may no longer perform well. In such a case, a different set $\set'$ may be used to match the image's direction (cf.~Remark 1 after Definition~\ref{def:cosetsum} in \S\ref{subS:cosetsumrefinementmasks}). For example, if the dominant direction is $-60^{\circ}$ from the positive $x$-axis, then since $\tan({-60}^{\circ})=-\sqrt{3}\approx\frac{2}{(-1)}$, the coset representative $(-1,2)$ can be used in place of $(1,0)$ in the default nonzero coset representatives $\set'=\{(1,0),(0,1),(1,1)\}$.

For many images, it may not be possible to match the directions of the image with the directions of the coset sum. In order to see how the coset sum system would perform for these images in comparison with the tensor product system, we apply the two systems to the test image, ``wood"\footnote{This image is produced by overlying 5 rotated versions of the image ``wood.000", which is used to generate the image ``wood45'' in Figure \ref{figure:wood45}(a), and taking its central part (of size $256 \times 256$).} in Figure~\ref{figure:wood}(a). It has $5$ different directional content ($25^\circ, 60^\circ$, $95^\circ$, $130^\circ$ and $165^\circ$ from the positive $x$-axis), and it is impossible for us to choose the nonzero coset representatives that match all the directions presented in the image. The reconstructed image using coset sum with the default coset representatives shows sharper edges along the directions near $45^\circ$, such as $60^\circ$ and $25^\circ$, and better visual quality than those of tensor product system, and has a slightly higher PSNR (see Figure \ref{figure:wood}(a)(c)(d)). For the directions that are significantly different from the preferred directions of the coset sum, such as $130^\circ$, it does not show sharp edges anymore, but the reproduced image using tensor product does not show sharp edges either. We found that the overall PSNR result of this image is similar to that of ``part of lena'': the reconstructed image by coset sum has slightly higher PSNR as long as the percentage of retained coefficients is not too large (see Figure \ref{figure:wood}(b) for the range 0.5\%-20\%), but for higher percentage it showed either comparable or slightly worse performance.

We notice that in the reconstructed image in Figure \ref{figure:wood}(d) certain directions are pronounced more strongly than others despite that the original image in Figure \ref{figure:wood}(a) does not have that characteristic. This is due to the lack of rotational symmetry (\cite{RRS}) of the coset sum refinable functions with the default coset representatives (cf. Figure \ref{figure:s4u4refinement} and \ref{figure:cosetsum_fr}(a)). A remedy for this can be obtained by choosing a set $\set'$ that gives (roughly) equi-angled directions.  For example, when $\dm=2$, by setting $\set'=\{(1,1),(-4,1),(1,-4)\}$, a rotational symmetry can be roughly achieved as it gives equi-angled directions $45^\circ, 165^\circ$, and $285^\circ$ from the positive $x$-axis. However, in general it is not easy to overcome the lack of rotational symmetry of the coset sum refinable functions. For example, if we use the above non-default choice of $\set'$ for $\mathcal{C}_2[S_4]$ and $\mathcal{C}_2[U_4]$ in Example 4, our computation shows that the refinable function associated with $\mathcal{C}_2[S_4]$ is still in $L^2(\RR^2)$, but the one associated with $\mathcal{C}_2[U_4]$ is not. Therefore obtaining coset sum refinable functions that are in $L^2(\Rd)$ with rotational symmetry may not be always possible even for the case of $\dm=2$.

As a passing remark, we make a brief comment on comparison to the curvelet system \cite{CaDo}, which is a state-of-the-art system for representing 2-D and 3-D data effectively using their geometric structure. Before presenting the image experiments using curvelets, we note that any comparison between the curvelet system and the coset sum wavelet system should be made with care as they are very different in nature. For one thing, the curvelet system is not a wavelet system constructed by using a method that works for any multi-D, which is our main interest in this paper. Besides, the curvelet system is highly redundant and its fast algorithm is slower than that of the tensor product and the coset sum system. 

With all these in mind, we perform the curvelet transform to the above test images using the Matlab implementation ({\it fdct\_wrapping.m} for decomposition and {\it ifdct\_wrapping.m} for reconstruction) of 2-D discrete curvelets (\cite{CDDY}) in CurveLab Toolbox \cite{CurveLabToolbox}. Even after fixing the decomposition level and the percentage of retained coefficients, there are still some parameters to be chosen in the curvelet codes, and the PSNR of reconstructed images is quite sensitive to the choice of these parameters. For reconstructed images using curvelet system with the default parameter setting, the PSNR is either between the PSNR of tensor product and that of coset sum (``wood45" and ``wood"), or slightly lower than the PSNR of tensor product (``lena") (see Figure \ref{figure:curvelet}). In terms of the visual quality, the curvelet system is superior to the other two systems in both capturing different directional content and keeping rotational symmetry in an image (see Figure \ref{figure:curvelet}(c)), but it may add some strong directional artifacts to the reconstructed image (see Figure \ref{figure:curvelet}(a)(b)). We conclude that a  complete comparison between the coset sum system and the curvelet system requires more thorough study on them.

\section{Summary and outlook}
\label{S:summary}
In this paper we presented the coset sum as an alternative method to the tensor product in constructing decomposable multivariate refinement masks.
The decomposable refinement mask constructed by coset sum can be written as the sum, instead of the product, of the univariate refinement masks.
We showed that the coset sum can provide many important features of the tensor product, such as preserving the biorthogonality of the univariate refinement masks and the availability of a wavelet system with fast algorithms.

Since the coset sum provides a way to obtain a pair of multivariate biorthogonal refinement masks, it can be combined with any method for finding wavelet masks to construct a (MRA-based biorthogonal) multivariate wavelet system.
There has been only limited progress in a systematic construction of non-tensor based multivariate wavelet systems. The coset sum adds a new opportunity to this end.

By specifying wavelet masks as described in \S\ref{subS:cosetsumwaveletsystems}, we constructed a particular class of coset sum wavelet systems that can be associated with fast algorithms. Such algorithms are referred to as fast coset sum wavelet algorithms. 

The fast tensor product wavelet algorithm has linear complexity, but the constant in the complexity bound increases as the spatial dimension increases. On the other hand, the constant in the (linear) complexity bound for the fast coset sum wavelet algorithm is independent of the dimension. Thus, when the spatial dimension is high, the coset sum wavelet algorithm can be {\it faster} than the tensor product wavelet algorithm. 

Coset sum is not necessarily the only alternative to the tensor product. Rather, despite of its limitations in processing images, its existence with desirable features suggests that it may be worthwhile to develop and practice alternative methods to the tensor product for constructing multivariate wavelet systems.

\appendix
\subsection{Proof of Theorem~\ref{thm:main}}
\label{subS:proofofThm1}
\subsubsection{Proof of part (a)}
Suppose $H$ and $h$ are the filters associated with masks $R$ and $\mathcal{C}_\dm[R]$. If $R$ is interpolatory, it is straightforward to show that $\mathcal{C}_\dm[R]$ is interpolatory. If $\mathcal{C}_\dm[R]$ is interpolatory, by (\ref{eq:interpfilter}), $h(0)=1$, and $h(k)=0$ if $k\in2\Zd\bks 0$. Then by (\ref{eq:filter}), $H(0)=1$. Moreover, $H(K)=0$ at all other even points, because if $H(K)\neq0$ at some even point $K\in2\ZZ\bks0$, then $h(k)=H(K)\neq0$ at $k=K\nu\in2\Zd\bks 0$, which contradicts to that $\mathcal{C}_\dm[R]$ is interpolatory. Therefore $R$ is also interpolatory.

\subsubsection{Proof of part (b)}
Without loss of generality, we may assume $\tilde{R}$ is interpolatory. We want to show that, $\mathcal{C}_\dm[R]$ and $\mathcal{C}_\dm[\tilde{R}]$ are biorthogonal if and only if $R$ and $\tilde{R}$ are biorthogonal.

    Let $R^{o}:=(R-R(\cdot+\pi))/2$ and $R^{e}:=(R+R(\cdot+\pi))/2$ be the odd and even
parts of $R$, respectively, and let $\tilde{R}^{o}$ be the odd part of $\tilde{R}$. Since $\tilde{R}$ is interpolatory, the even
part of $\tilde{R}$ is the constant $1/2$. It is easy to check $\forall\ome_1\in\TT$
\beaN
&&\overline{R^{o}(\ome_1)}\tilde{R}^{o}(\ome_1)={\disp 1\over \disp 2}-{\disp 1\over \disp
2}\overline{R^{e}(\ome_1)}\\
&&\Longleftrightarrow R\:\mbox{and}\:\tilde{R}\:\mbox{are biorthogonal}.
\eeaN
Here, as before, the overline is used to denote the
complex conjugate.

We will also need the following identities:
\be
\label{eq:charrunpiset}
  \sum_{\gam\in\pi\set}e^{-i\nu\cdot\gam}
  =\cases{2^\dm,& if $\nu=0$,\cr
          0,              & if $\nu\in\set'$.\cr}
\ee
Then from the definition of the coset sum (cf.~Definition~\ref{def:cosetsum}, (\ref{eq:simpleform}) and (\ref{eq:specialform})), biorthogonal condition (\ref{eq:refinebiormask}), and the above identities (\ref{eq:charrunpiset}), we have
\beaN
&{ }&\:\mathcal{C}_n[R]\: \mbox{and}\: \mathcal{C}_n[\tilde{R}] \:\mbox{are biorthogonal}\\
\Longleftrightarrow&&
\sum_{\gam\in\pi\set}(\overline{\mathcal{C}_\dm[R]}\mathcal{C}_\dm[\tilde{R}])(\ome+\gam)=1
,\quad\forall\ome\in\Td\\
\Longleftrightarrow&&
\left({\disp 1\over \disp 2^{\dm-1}}\right)^2\sum_{\gam\in\pi\set}
\left(-2^{\dm-1}+\sum_{\nu\in\set}\overline{R((\ome+\gam)\cdot\nu)}\right)\cdot\\
&&\left({\disp 1\over \disp
2}+\sum_{\tilde\nu\in\set'}\left(\tilde{R}((\ome+\gam)\cdot\tilde\nu)-{\disp 1\over
\disp 2}\right)\right)=1,\quad\forall\ome\in\Td\\
\Longleftrightarrow&&
\sum_{\gam\in\pi\set}
\Bigg(1-2^{\dm-1}+\sum_{\nu\in\set'}e^{i\gam\cdot\nu}\overline{R^{o}(\ome\cdot\nu)}
+\sum_{\nu\in\set'}\overline{R^{e}(\ome\cdot\nu)}\Bigg)\cdot\\
&&
\left({\disp 1\over \disp 2}
+\sum_{\tilde\nu\in\set'}e^{-i\gam\cdot\tilde\nu}\tilde{R}^{o}(\ome\cdot\tilde\nu)\right)=(2^{\dm-1})^2,\quad\forall\ome\in\Td\\
\Longleftrightarrow&&
2^{\dm-1}(1-2^{\dm-1})+2^{\dm-1}\sum_{\nu\in\set'}\overline{R^{e}(\ome\cdot\nu)}\\
&&+2^\dm\sum_{\nu\in\set'}\overline{R^{o}(\ome\cdot\nu)}\tilde{R}^{o}(\ome\cdot\nu)=(2^{\dm-1})^2,\quad\forall\ome\in\Td\\
\Longleftrightarrow&&
\overline{R^{o}(\ome_1)}\tilde{R}^{o}(\ome_1)={\disp 1\over \disp 2}-{\disp 1\over \disp
2}\overline{R^{e}(\ome_1)},\quad\forall\ome_1\in\TT.
\eeaN
Therefore, $\mathcal{C}_\dm[R]$ and $\mathcal{C}_\dm[\tilde{R}]$ are biorthogonal if and only if $R$ and $\tilde{R}$ are biorthogonal.

\subsubsection{Proof of part (c)}
Let $R$ be a univariate interpolatory refinement mask with accuracy number $m$. First let us prove the accuracy number of $\mathcal{C}_\dm[R]$ is at least $m$.
Since $R$ has accuracy number $m$,
\be
\label{eq:Ratpi}
(D^kR)(\pi)=0,\; \forall 0\le k\le m-1, \:\hbox{ and } \: (D^{m}R)(\pi)\ne 0.
\ee
Furthermore, since $R$ is interpolatory, $1-R(\ome)=R(\ome+\pi)$ holds for all $\ome\in\TT$. Hence $(D^k(1-R))(0)=(D^kR)(\pi)$ for all $k\in\NN_0:=\NN\cup\{0\}$. Thus $1-R$ has a zero of order $m$ at the origin, i.e.
\bea
R(0)&{\,=\,}&1\nonumber\\
(D^kR)(0)&=&0,\quad \forall 1\le k\le m-1\label{eq:Ratzero}\\
(D^mR)(0)&\ne& 0.\nonumber
\eea

Now consider the $\dm$-D refinement mask $\mathcal{C}_\dm[R]$. The accuracy number of $\mathcal{C}_\dm[R]$ is at least one, i.e. $\mathcal{C}_\dm[R](\gam)=0$, for all $\gam\in\pi\set'$. To see this, we need the dual identities of (\ref{eq:charrunpiset}):
\be
\label{eq:charrunset}
\sum_{\nu\in\set}e^{-i\nu\cdot\gam}
  =\cases{2^\dm,& if $\gam=0$,\cr
          0,             & if $\gam\in\pi\set'$.\cr}
\ee
From (\ref{eq:charrunset}), we can read off
\be
\label{eq:congruent}
\#\{\nu\in\set':\gam\cdot\nu\equiv\pi\,(\mod 2\pi\ZZ)\}=2^{\dm-1},
\ee
for all $\gam\in\pi\set'$. In particular, the left-hand side of
(\ref{eq:congruent}) is independent of
$\gam$. We then have for any $\gam\in\pi\set'$
\beaN
&&2^{\dm-1}\mathcal{C}_\dm[R](\gam)= -2^{\dm-1}
+\sum_{\nu\in\set}R(\gam\cdot\nu)\\
&=&1-2^{\dm-1}
+\sum_{\{\nu\in\set':\gam\cdot\nu\equiv 0\}}R(\gam\cdot\nu)
+\sum_{\{\nu\in\set':\gam\cdot\nu\equiv\pi\}}R(\gam\cdot\nu)\\
&=&0,
\eeaN
where $\equiv$ in the second line is used to denote congruence in modulo $2\pi\ZZ$, and the last equality is from the conditions $R(0)=1$, $R(\pi)=0$ and the identity
(\ref{eq:congruent}).
Furthermore, for all $\gam\in\pi\set'$ and for all $\mu\in\NN_0^\dm$ with $1\le \onenorm{\mu}\le m-1$ ($\onenorm{\mu}:=\mu_1+\cdots+\mu_\dm$)
\beaN
&&(D^\mu\mathcal{C}_\dm[R])(\gam)
={1\over 2^{\dm-1}}\sum_{\nu\in\set'}(D^\mu[R(\ome\cdot\nu)])_{|\ome=\gam}\\
&&={1\over 2^{\dm-1}}\sum_{\nu\in\set'}\left(\prod_{j=1}^\dm
\nu_j^{\;\mu_j}\right)
(D^{\onenorms{\mu}}R)(\gam\cdot\nu)=0,\nonumber
\eeaN
where the last equality is from the identities (\ref{eq:Ratpi}) and
(\ref{eq:Ratzero}). Therefore the accuracy number of $\mathcal{C}_\dm[R]$ is at least $m$.

Next we prove the accuracy number of $\mathcal{C}_\dm[R]$ is exactly $m$ by contradiction. Suppose the accuracy number of $\mathcal{C}_\dm[R]$ is $m+l$ with $l\ge 1$. Then
\beaN
&&(D^\mu\mathcal{C}_\dm[R])(\gam)=0,\\
\forall \gam\in\pi\set' \hbox{ and } &&\forall \mu\in\NN_0^\dm \hbox{ with } 0\le\onenorm{\mu}\le m+l-1.
\eeaN
Since the univariate interpolatory $R$ and the multivariate interpolatory
$\mathcal{C}_\dm[R]$ are connected as follows:
$$
%\label{eq:Han}
R(\ome)=\mathcal{C}_\dm[R](\ome,0,\cdots,0),\quad \forall \ome\in\TT,
$$
we have
$(D^kR)(\pi)=D^{(k,0,\cdots,0)}\mathcal{C}_\dm[R](\pi,0,\cdots,0)=0$ for
all $0\le k\le m+l-1$. Hence the accuracy number of $R$ is at least $m+l$, which contradicts to the given assumption. Therefore the accuracy number
of $\mathcal{C}_\dm[R]$ has to be $m$.

\subsection{Proof of Theorem~\ref{thm:cosetsumwavelet}}
\label{subS:proofofThm2}
In this subsection we prove Theorem~\ref{thm:cosetsumwavelet}. In the proof we use the concepts of
Compression-Alignment-Prediction (CAP) and Compression-Alignment-Modified-Prediction (CAMP)
\cite{HR1}. CAMP is a variant of CAP, and CAP is a generalization of the Laplacian
pyramid \cite{BA}. In particular, CAP without alignment operator is the same as Laplacian pyramid. It is well known that Laplacian pyramid has a trivial reconstruction algorithm of reversing the steps in its decomposition algorithm.
Both CAP and CAMP are originally designed for the redundant wavelet construction, and CAMP is introduced in order to achieve a better
space localization than CAP.

Given $\tau:=\mathcal{C}_\dm[S]$, $\tau^\du:=\mathcal{C}_\dm[U]$ with interpolatory $U$, and $t_\nu(\ome):=e^{-i\nu\cdot\ome}\overline{U(\ome\cdot\nu+\pi)}$, $\ome\in\Td$, $\nu\in \set'$, we want to show that there exist dual wavelet masks $t_\nu^\du(\ome)$ such that $(\tau,(t_\nu)_{\nu\in\set'})$ and $(\tau^\du,(t_\nu^\du)_{\nu\in\set'})$ satisfy the MUEP conditions in (\ref{eq:mUEP}).

To show this, first let us construct another pair of wavelet masks $(\tau_\nu)_{\nu\in\set}$ and dual wavelet masks $(\tau_\nu^\du)_{\nu\in\set}$, which we know for sure satisfy the MUEP conditions with $\tau$ and $\tau^\du$.

First extend the definition of $t_\nu$ by defining $t_0$:
$$
t_\nu(\ome):=\cases{{\disp1\over\disp2}(1-\tau(\ome)),&if $\nu=0$,\cr e^{-i\nu\cdot\ome}\overline{U(\ome\cdot\nu+\pi)},&if $\nu\in\set'$.\cr}
$$
Then from \cite{HR1} we know that
\be
t_\nu(\ome)=2^{{n\over2}-1}\cdot t_{-\nu}^{CAMP}(\ome),\quad\nu\in\set,
\label{eq:t_nu}
\ee
where $t_\nu^{CAMP}$ is the CAMPlet mask in Section 2.3 of \cite{HR1}.

Furthermore by comparing the CAPlet masks in Lemma 2.2 of \cite{HR1} with the CAMPlet masks, it is easy to see that they are related as
\be
t_\nu^{CAP}(\ome)-t_\nu^{CAMP}(\ome)=\cases{0,&if $\nu=0$,\cr
f_{-\nu}(\ome)t_0^{CAMP}(\ome),&if $\nu\in\set'$,\cr}
\label{eq:CAPCAMP}
\ee
where $f_\nu(\ome)=e^{-i\nu\cdot\ome}\sum_{\gam\in\pi\set}{e^{-i\nu\cdot\gam}\overline{\tau^\du(\ome+\gam)}}$. Here it is necessary to point out that
$f_\nu$ is $\pi$-periodic, i.e. $f_\nu(\ome+\gam)=f_\nu(\ome)$, for any $\gam\in\pi\set$.

Now define $(\tau_\nu)_{\nu\in\set}$
\be
\tau_\nu(\ome):=2^{{n\over2}-1}t_{-\nu}^{CAP}(\ome),\quad\nu\in\set.
\label{eq:tau_nu}
\ee
Then since CAP without alignment operator is the same as Laplacian pyramid, and Laplacian pyramid has a trivial reconstruction, we know that with
$$
\tau_\nu^\du(\ome):=\cases{2^{1-n},&if $\nu=0$,\cr
2^{1-n}e^{-i\nu\cdot\ome},&if $\nu\in\set'$,\cr}
$$
$(\tau,(\tau_\nu)_{\nu\in\set})$ and $(\tau^\du,(\tau_\nu^\du)_{\nu\in\set})$ satisfy the MUEP conditions.

Next, we start from the MUEP conditions of $(\tau,(\tau_\nu)_{\nu\in\set})$ and $(\tau^\du,(\tau_\nu^\du)_{\nu\in\set})$ to find our dual wavelet masks $t_\nu^\du$. To do that, we need three more identities. The first one is a simple observation that can be obtained from (\ref{eq:t_nu}), (\ref{eq:CAPCAMP}) and (\ref{eq:tau_nu}):
\be
\tau_\nu(\ome)-t_\nu(\ome)=\cases{0,&if $\nu=0$,\cr f_\nu(\ome)t_0(\ome),&if $\nu\in\set'$.\cr}
\label{eq:tau_nu-t-nu}
\ee
The second one can be derived from the interpolatory property of $\tau^\du$ and the identities (\ref{eq:charrunset}):
\be
\tau^\du_0(\ome)+\sum_{\nu\in\set'}\overline{f_\nu(\ome)}\tau^\du_\nu(\ome)=2\tau^\du(\ome).
\label{eq:f_nu}
\ee
After defining $g_\nu(\ome):=e^{-i\nu\cdot\ome}\sum_{\gam\in\pi\set}e^{-i\nu\cdot\gam}\overline{\tau(\ome+\gam)}$, the third identity:
\be
t_0(\ome)+2^{-n}\sum_{\nu\in\set'}t_\nu(\ome)\overline{g_\nu(\ome)}=0
\label{eq:g_nu}
\ee
can be shown from the biorthogonality between $\tau$ and $\tau^\du$ and the identities (\ref{eq:charrunset}).
Finally from the above identities (\ref{eq:tau_nu-t-nu}), (\ref{eq:f_nu}) and (\ref{eq:g_nu}), 
with 
$$\delta_{\gamma 0}
:=\cases{
1,&if $\gam=0$,\cr 0,&if $\gam\in \pi\set'$,\cr}$$
we get
\beaN
&{}&\delta_{\gam0}\\
&=&\overline{\tau(\ome+\gam)}\tau^\du(\ome)+\overline{\tau_0(\ome+\gam)}\tau_0^\du(\ome)
+\sum_{\nu\in\set'}\overline{\tau_\nu(\ome+\gam)}\tau_\nu^\du(\ome)\\
&=&\overline{\tau(\ome+\gam)}\tau^\du(\ome)+\overline{t_0(\ome+\gam)}\tau_0^\du(\ome)\\
& &+\sum_{\nu\in\set'}
\overline{f_\nu(\ome+\gam)t_0(\ome+\gam)+t_\nu(\ome+\gam)}\tau_\nu^\du(\ome)\\
&=&\overline{\tau(\ome+\gam)}\tau^\du(\ome)\\
& &+\overline{t_0(\ome+\gam)}\Bigg(\tau_0^\du(\ome)+\sum_
{\nu\in\set'}\overline{f_\nu(\ome+\gam)}\tau_\nu^\du(\ome)\Bigg)\\
& &+\sum_{\nu\in\set'}\overline{t_\nu
(\ome+\gam)}\tau_\nu^\du(\ome)\\
&=&\overline{\tau(\ome+\gam)}\tau^\du(\ome)+\overline{t_0(\ome+\gam)}2\tau^\du(\ome)+\sum_{\nu\in\set'}
\overline{t_\nu(\ome+\gam)}\tau_\nu^\du(\ome)\\
&=&\overline{\tau(\ome+\gam)}\tau^\du(\ome)\\
& &+\left(\overline{t_0(\ome+\gam)}+2^{-\dm}\sum_{\nu\in
\set'}\overline{t_\nu(\ome+\gam)}g_\nu(\ome+\gam)\right)2\tau^\du(\ome)\\
& &-2^{-\dm}\sum_{\nu\in\set'}\overline{t_\nu(\ome+\gam)}g_\nu(\ome+\gam)2\tau^\du(\ome)\\
& &+\sum_{\nu\in\set'}\overline{t_\nu
(\ome+\gam)}\tau_\nu^\du(\ome)\\
&=&\overline{\tau(\ome+\gam)}\tau^\du(\ome)
-2^{1-\dm}\sum_{\nu\in\set'}\overline{t_\nu(\ome+\gam)}g_\nu(\ome+\gam)\tau^\du(\ome)\\
& &+\sum_{\nu\in\set'}\overline{t_\nu
(\ome+\gam)}\tau_\nu^\du(\ome)\\
&=&\overline{\tau(\ome+\gam)}\tau^\du(\ome)\\
& &+\sum_{\nu\in\set'}\overline{t_\nu(\ome+\gam)}
\left(-2^{1-\dm}g_\nu(\ome)\tau^\du(\ome)+\tau_\nu^\du(\ome)\right).
\eeaN

Therefore, by letting $t_\nu^\du:=-2^{1-n}g_\nu\tau^\du+\tau_\nu^\du$, we find the dual wavelet masks
$t_\nu^\du$, $\nu\in\set'$, such that $(\tau,(t_\nu)_{\nu\in\set'})$ and $(\tau^\du,(t_\nu^\du)_{\nu\in\set'})$ satisfy the MUEP conditions.

\subsection{Proof of the identity in Step {\tt (iii)} of the coset sum algorithm in \S\ref{subS:Algorithms}}
\label{subS:stepiii}
In this subsection we verify the identity in Step {\tt (iii)} of Reconstruction Algorithm in \S\ref{subS:Algorithms}. We use the same notation as in the algorithm. In particular, $G$ and $H$ are univariate refinement filters associated with biorthogonal refinement masks $S$ and $U$, respectively, and $H$ is interpolatory. 

From Step {\tt (i)}  of the algorithm in \S\ref{subS:Algorithms}, we know that, with $a_G=2^\dm-(2^\dm-1)(2-G(0))$,
\bea
a_G\,y_j(2k)&=&2^\dm y_{j-1}(k)-\sum_{\nu\in\set'}\sum_{L\in\ZZ\bks0}G(L)y_j(2k+L\nu)\nonumber\\
&=&2^\dm y_{j-1}(k)-\sum_{\nu\in\set'}\sum_{L \equiv 1}G(L)y_j(2k+L\nu)\nonumber\\
&{}&-\sum_{\nu\in\set'}\sum_{L\equiv 0, L\ne 0}G(L)y_j(2k+L\nu)
\label{eq:y2k}
\eea
where $\equiv$ is used to denote congruence in modulo $2\ZZ$. Since the masks $S$ and $U$ are biorthogonal, from (\ref{eq:refinebiormask}) and the connection between the filter and the mask, it is easy to see that the associated filters $G$ and $H$ satisfy the following condition:
\beaN
\sum_{m\in\ZZ}G(L+m)H(m)
=\cases{
0,&if $L\equiv 0, L\ne 0$,\cr 2,&if $L=0$.\cr}
\eeaN
Combining this with the fact that $H$ is interpolatory leads to
$$\sum_{m\equiv1}G(L+m)H(m)
=\cases{
0-G(L),&if $L\equiv 0, L\ne 0$,\cr 2-G(0),&if $L=0$.\cr}$$
From this and the change of variables, we see that 
\beaN
&&\sum_{\nu\in\set'}\sum_{L\equiv 0, L\ne 0}G(L)y_j(2k+L\nu)\\
&=&\sum_{\nu\in\set'}\sum_{L\equiv 0, L\ne 0}\left(0-\sum_{m\equiv1}G(L+m)H(m)\right)y_j(2k+L\nu)\\
&=&-\sum_{\nu\in\set'}\sum_{m\equiv1}\sum_{L\equiv 0}G(L+m)H(m)y_j(2k+L\nu)\\
&&+\sum_{\nu\in\set'}\sum_{m\equiv1}G(m)H(m)y_j(2k)
%\qquad (n=L+m)
\\
&=&-\sum_{\nu\in\set'}\sum_{m\equiv1}\sum_{n\equiv 1}G(n)H(m)y_j(2k+(n-m)\nu)\\
&&+(2^\dm-1)(2-G(0))y_j(2k)\\
&=&-\sum_{\nu\in\set'}\sum_{L\equiv1}\sum_{m\equiv1}G(L)H(m)y_j(2k+(L-m)\nu)\\
&&+(2^\dm-1)(2-G(0))y_j(2k)
\eeaN
By substituting this result to (\ref{eq:y2k}) and solving for $y_j(2k)$, we obtain
%\beaN
%&&(1+{1\over a_G}(2^\dm-1)(2-G(0)))y_j(2k)\\
%=&&{2^\dm \over a_G}y_{j-1}(k)-{1\over a_G}\sum_{\nu\in\set'}\sum_{L\mbox{\small \ is odd}}G(L)y_j(2k+L\nu)\\
%&&-{1\over a_G}\sum_{\nu\in\set'}\sum_{m\equiv1}\sum_{n\mbox{\small \ is odd}}G(n)H(m)y_j(2k+(n-m)\nu\\
%\eeaN
%Hence,
%\beaN
%2^\dm y_j(2k)&=&2^\dm y_{j-1}(k)-\sum_{\nu\in\set'}\sum_{L\equiv1}G(L)\{y_j(2k+L\nu)\\
%&&-\sum_{m\equiv1}H(m)y_j(2k+(L-m)\nu)\}\\
%&=&2^\dm y_{j-1}(k)-2\sum_{\nu\in\set'}\sum_{L\in\ZZ}G(2L+1)w_{\nu,j-1}(k+L\nu)
%\eeaN
%Therefore,
\beaN
y_j(2k)=y_{j-1}(k)-{1\over 2^{\dm-1}}\sum_{\nu\in\set'}\sum_{L\in\ZZ}G(2L+1)w_{\nu,j-1}(k+L\nu)
\eeaN
as desired.

\section*{Acknowledgment}
The authors thank the editor and referees for suggestions that improved the clarity and accessibility of this article. 

\bibliographystyle{IEEEtran} % use IEEEtran.bst style
\bibliography{IEEEabrv,mybibfile}

\newpage

\begin{IEEEbiographynophoto}{Youngmi Hur} (M'11) received
the B.S. and M.S. degrees in mathematics from Korea Advanced Institute of Science and Technology, Daejeon, South Korea, in 1997 and 1999, respectively, and the Ph.D. degree in mathematics from University of Wisconsin at Madison in 2006.

From 2006 to 2008, she was a C.L.E. Moore instructor at the Department of Mathematics of the Massachusetts Institute of Technology, Cambridge, MA. Since 2008, she is an Assistant Professor at the Department of Applied Mathematics and Statistics of the Johns Hopkins University, Baltimore, MD. Her research interests are in the field of applied and computational harmonic analysis including wavelets and their applications.
\end{IEEEbiographynophoto}

\vfill

\begin{IEEEbiographynophoto}{Fang Zheng} is currently a Ph.D candidate in the Department of Applied Mathematics and Statistics at Johns Hopkins University, under the supervision of Prof. Youngmi Hur. Her research interests include multivariate wavelet system and its applications. Fang Zheng received a Bachelors degree in Mathematics from Beijing Forestry University, China in 2008, and a Masters degree in Applied Mathematics and Statistics from Johns Hopkins University in 2010.
\end{IEEEbiographynophoto}
\vfill

\end{document}